\documentclass[10pt]{article}
\pdfoutput=1
\usepackage{amsmath}
\usepackage{amsfonts}
\usepackage{amssymb}
\usepackage{amsthm}
\usepackage{graphicx}

\newtheorem{theorem}{Theorem}[section]

\newtheorem{corollary}[theorem]{Corollary}

\newtheorem{proposition}[theorem]{Proposition}
\theoremstyle{definition}
\newtheorem{definition}[theorem]{Definition}
\theoremstyle{remark}
\newtheorem{example}[theorem]{Example}
\newtheorem{remark}[theorem]{Remark}
\numberwithin{equation}{section}
\begin{document}

\title{Analytic approximation of transmutation operators and applications to highly accurate solution of spectral problems}
\author{Vladislav V. Kravchenko and Sergii M. Torba \\
{\small Departamento de Matem\'{a}ticas, CINVESTAV del IPN, Unidad Quer\'{e}%
taro, }\\
{\small Libramiento Norponiente No. 2000, Fracc. Real de Juriquilla, Quer%
\'{e}taro, Qro. C.P. 76230 MEXICO}\\
{\small  e-mail: vkravchenko@math.cinvestav.edu.mx, storba@math.cinvestav.edu.mx
\thanks{Research was supported by CONACYT, Mexico via the project 166141.}}}
\maketitle

\begin{abstract}
A method for approximate solution of spectral problems for Sturm-Liouville
equations based on the construction of the Delsarte transmutation operators
is presented. In fact the problem of numerical approximation of solutions
and eigenvalues is reduced to approximation of a primitive of the potential
by a finite linear combination of generalized wave polynomials\ introduced
in \cite{KKTT}, \cite{KT Transmut}. The method allows one to compute both
lower and higher eigendata with an extreme accuracy.
\end{abstract}

\section{Introduction}

Solution of Sturm-Liouville equations and of a wide range of related direct
and inverse spectral problems is at the core of modern mathematical physics
and its numerous applications. Since the work of J. Fourier on the theory of
the heat and his method of separation of variables the properties and
methods for solving different kinds of Sturm-Liouville spectral problems
were studied in thousands of publications. One of the important mathematical
tools for approaching problems related to Sturm-Liouville equations was
introduced in 1938 by J. Delsarte \cite{Delsarte} and called \cite%
{DelsarteLions1956} the transmutation operator. It relates two linear
differential operators and allows one to transform a more complicated
equation into a simpler one. Nowadays the transmutation operator is widely
used in the theory of linear differential equations (see, e.g., \cite%
{BegehrGilbert}, \cite{Carroll}, \cite{HrynivMykytyuk2003}, \cite{LevitanInverse}, \cite{Marchenko},
\cite{Sitnik}, \cite{Trimeche}). Very often in literature the transmutation
operators are called the transformation operators. It is well known that
under certain regularity conditions the transmutation operator transmuting
the operator $A=-\frac{d^{2}}{dx^{2}}+q(x)$ into $B=-\frac{d^{2}}{dx^{2}}$
is a Volterra integral operator with good properties. Its integral kernel
can be obtained as a solution of a certain Goursat problem for the
Klein-Gordon equation with a variable coefficient. In spite of their
attractive properties and importance there exist very few examples of the
transmutation kernels available in a closed form (see \cite{KrT2012}).

In the recent work \cite{KT Transmut} it was observed that the kernel of the
transmutation operator relating the Schr\"{o}dinger operator $A$ with $B$ is a complex
component of a bicomplex-valued pseudoanalytic function of a hyperbolic
variable \cite{APFT}, a solution of a hyperbolic Vekua equation of a special
form. The other component of that pseudoanalytic function is the
transmutation kernel corresponding to a Darboux associated Schr\"{o}dinger
operator \cite{KrT2012}. This observation combined with some new results
concerning such hyperbolic pseudoanalytic functions allowed us to obtain
\cite{KT Transmut} a new and extremely convenient representation for the
transmutation kernel in terms of so-called generalized wave polynomials \cite%
{KKTT}.

In the present work we develop this result into a practical method for
solving a wide spectrum of initial value, boundary value and spectral
problems for the one-dimensional Schr\"{o}dinger equation $Au=\lambda u$.
The mentioned above convenience of the representation of the kernel consists
in the fact that with respect to the variable of integration the kernel
results to be a polynomial. Since to obtain solutions of the Schr\"{o}dinger
equation the transmutation operator is applied to the functions $\sin \sqrt{%
\lambda }t$ and $\cos \sqrt{\lambda }t$, solutions of the simplest such
equation $Bv=\lambda v$, all the involved integrals are calculated
explicitly. The coefficients of the polynomial are functions of
another independent variable $x$. Due to the developed theory they can be found as a result of approximation of the pair of functions $g_{1}(x)=\frac{h}{2}+\frac{1}{4}\int_{0}^{x}q(s)ds$ and $g_{2}(x)=\frac{1}{4}\int_{0}^{x}q(s)ds$ in terms of a specially constructed family of functions. Here $h$ is a constant defined below. The family of functions is the family of traces of the generalized wave polynomials on the lines $x=t$ and $x=-t$ in the plane $(x,t)$.

Thus, the method developed and presented in this paper for solving spectral
problems for the Schr\"{o}dinger operator consists in the following steps.

\begin{enumerate}
\item Construction of a particular solution for the equation $Au=0$.

\item Construction of certain recursive integrals which serve for
calculating the family of traces of the generalized wave polynomials.

\item Approximation of the pair of functions $g_{1}$ and $g_{2}$ by a linear
combination of the traces.

\item Computation of the solution of the Schr\"{o}dinger equation or of a
characteristic function of the spectral problem in a domain of interest as a
function of the spectral parameter $\lambda $.

\item Localization of the eigenvalues as zeros of the characteristic
function.
\end{enumerate}

Several advantages of the method should be emphasized.

\begin{itemize}
\item The error of the computed eigendata does not increase for higher
eigenvalues. One can compute, e.g., the $1000$th eigenvalue and
eigenfunction with roughly the same accuracy as the first ones.

\item A quite simple a-priori control of the accuracy of the computed eigendata is available. As we show, the accuracy of the computed
eigendata is entirely linked to the accuracy of approximation of the
functions $g_{1}$ and $g_{2}$ on step 3. This error of approximation is
easily calculated. If it is inadmissible, more approximating functions
should be taken.

\item The method works equally well in the case of complex-valued
coefficients and spectral parameter dependent boundary conditions.

\item The method allows one to obtain highly accurate eigendata. For
example, for a standard test problem (the Paine problem \cite{Pryce}) we
present eigendata corresponding to $\lambda _{1},\ldots,
\lambda _{10000}$ computed with the accuracy of the order $10^{-105}$. For
the Coffey-Evans problem well known for being extremely difficult for
numerical computation due to neatly clustered eigenvalues the achieved
accuracy is of the order $10^{-65}$. Note that even for $\lambda =0$ the
solution of the Coffey-Evans equation delivered by the built-in
Mathematica's function NDSolve was computed at best with an absolute error
greater than $61$.
\end{itemize}

The main results presented in this paper are aimed to give a rigorous
justification of the developed method. We prove that the central object of
the transmutation kernel representation, the set of traces of the
generalized wave polynomials is complete in the required functional spaces
and therefore the set can be used for approximating the functions $g_{1}$
and $g_{2}$. We obtain corresponding estimates for the accuracy of
approximation of the transmutation kernel. Further estimates concerning the
accuracy of the resulting approximation of solutions of the Schr\"{o}dinger
equation and the independence of the accuracy of largeness of $\lambda $ are
proved as well. We show that the same approximation coefficients obtained on
step 3 can be used to obtain a related approximation for the transmutation
kernel for the Darboux associated Schr\"{o}dinger operator. This is
important for considering problems involving derivatives in boundary
conditions.

The numerical experiments were performed in Mathematica with
multiple-precision arithmetic. The main reason for not restricting our
computations to the machine precision consists in a fuller exploration of
the method, its capabilities and features. For example, the results of
involved computations obtained with a high precision allow us to study more
completely the link between the accuracy of the approximation (step 3) and
of the resulting computed eigendata. This study reveals that the absolute
error of the approximation essentially coincides with the final error in the
eigendata and decays exponentially with respect to the number $N$ of traces used, which for practical purposes means that an accurate approximation
of the functions $g_{1}$ and $g_{2}$ by the system of traces of generalized
polynomials guarantees the computation of arbitrarily high eigendata with
the same accuracy. However the complete description of the class of potentials for which the approximation and eigendata errors decay exponentially and rigorous proofs of the observed relation between approximation errors and eigendata errors remain for the future research.

Due to the difficulties with construction of the kernels of transmutation operators there were very few attempts for their practical use in numerical solution of spectral problems. In this relation we mention the paper \cite{Boumenir2006} where certain analytic approximation formulas for the transmutation kernels were obtained. To our knowledge the present paper is a first publication offering an efficient and highly accurate (and to our opinion clearly promising) numerical algorithm based on the transmutation operators for solving spectral problems for the Schr\"{o}dinger operator. One of the direct applications requiring a large number of eigendata computed with a considerable and non-decreasing  accuracy arises from the Fourier method of separation of variables. According to the method the solution admits an analytic expression in the form of a series which in practice is known to be slowly convergent. The method presented here allows one to calculate partial sums of such series containing large numbers of terms in a fast and accurate manner.

In the next Section \ref{Section2} we recall some definitions and properties concerning the
transmutation operators. In Section \ref{Section3} we introduce the system of generalized
wave polynomials. In Section \ref{Section4} we prove the completeness of their traces in
appropriate functional spaces. In Section \ref{SectApprox} we construct the approximate
kernels of the transmutation operators and obtain corresponding estimates for
their accuracy. In Section \ref{Section6} we obtain the main result of the paper, the
formulas for approximate solutions of the Schr\"{o}dinger equation $Au=\lambda
u$ as well as for their derivatives and prove corresponding estimates for
their accuracy. Section \ref{Section7} is dedicated to the description of the algorithm and its numerical
implementation, the proof of the uniform error bounds for the approximate zeros of characteristic functions of Sturm-Liouville spectral problems as well as to the presentation of numerical results.

\section{Transmutation operators}\label{Section2}

We give a definition of a transmutation operator from \cite{KT Obzor} which
is a modification of the definition proposed by Levitan \cite{LevitanInverse}%
, adapted to the purposes of the present work. Let $E$ be a linear
topological space and $E_{1}$ its linear subspace (not necessarily closed).
Let $A$ and $B$ be linear operators: $E_{1}\rightarrow E$.

\begin{definition}
\label{DefTransmut} A linear invertible operator $T$ defined on the whole $E$
such that $E_{1}$ is invariant under the action of $T$ is called a
transmutation operator for the pair of operators $A$ and $B$ if it fulfills
the following two conditions.

\begin{enumerate}
\item Both the operator $T$ and its inverse $T^{-1}$ are continuous in $E$;

\item The following operator equality is valid
\begin{equation}
AT=TB  \label{ATTB}
\end{equation}
or which is the same
\begin{equation*}
A=TBT^{-1}.
\end{equation*}
\end{enumerate}
\end{definition}

Our main interest concerns the situation when $A=-\frac{d^{2}}{dx^{2}}+q(x)$%
, $B=-\frac{d^{2}}{dx^{2}}$, \ and $q$ is a continuous complex-valued
function. Hence for our purposes it will be sufficient to consider the
functional space $E=C[a,b]$ with the topology of uniform convergence and its
subspace $E_{1}$ consisting of functions from $C^{2}\left[ a,b\right] $. One
of the possibilities to introduce a transmutation operator on $E$ was
considered by Lions \cite{Lions57} and later on in other references (see,
e.g., \cite{Marchenko}), and consists in constructing a Volterra integral
operator corresponding to a midpoint of the segment of interest. As we begin
with this transmutation operator it is convenient to consider a symmetric
segment $[-b,b]$ and hence the functional space $E=C[-b,b]$. It is worth
mentioning that other well known ways to construct the transmutation
operators (see, e.g., \cite{LevitanInverse}, \cite{Trimeche}) imply imposing
initial conditions on the functions and consequently lead to transmutation
operators satisfying (\ref{ATTB}) only on subclasses of $E_{1}$. We introduce such transmutation operators below.

Thus, consider the space $E=C[-b,b]$. In \cite{CKT} and \cite{KrT2012} a
parametrized family of transmutation operators for the defined above $A$ and
$B$ was studied. Operators of this family can be realized in the form of the
Volterra integral operator
\begin{equation}
\mathbf{T}_{h} u(x)=u(x)+\int_{-x}^{x}\mathbf{K}(x,t;h)u(t)dt  \label{Tmain}
\end{equation}
where $\mathbf{K}(x,t;h)=\mathbf{H}\big(\frac{x+t}{2},\frac{x-t}{2};h\big)$,
$h$ is a complex parameter, $|t|\le|x|\le b$ and $\mathbf{H}$ is the unique
solution of the Goursat problem
\begin{equation}
\frac{\partial^{2}\mathbf{H}(u,v;h)}{\partial u\,\partial v}=q(u+v)\mathbf{H}%
(u,v;h),  \label{GoursatTh1}
\end{equation}%
\begin{equation}
\mathbf{H}(u,0;h)=\frac{h}{2}+\frac{1}{2}\int_{0}^{u}q(s)\,ds,\qquad \mathbf{%
H}(0,v;h)=\frac{h}{2}.  \label{GoursatTh2}
\end{equation}
If the potential $q$ is continuously differentiable, the kernel $\mathbf{K}$
itself is a solution of the Goursat problem
\begin{equation}  \label{GoursatKh1}
\left( \frac{\partial^{2}}{\partial x^{2}}-q(x)\right) \mathbf{K}(x,t;h)=%
\frac{\partial^{2}}{\partial t^{2}}\mathbf{K}(x,t;h),
\end{equation}
\begin{equation}  \label{GoursatKh2}
\mathbf{K}(x,x;h)=\frac{h}{2}+\frac{1}{2}\int_{0}^{x}q(s)\,ds,\qquad\mathbf{K%
}(x,-x;h)=\frac{h}{2}.
\end{equation}
If the potential $q$ is $n$ times continuously differentiable, the kernel $%
\mathbf{K}(x,t;h)$ is $n+1$ times continuously differentiable with respect
to both independent variables.

\begin{remark}
In the case $h=0$ the operator $\mathbf{T}_{h}$ coincides with the
transmutation operator studied in \cite[Chap. 1, Sect. 2]{Marchenko}. In
\cite{LevitanInverse}, \cite{Lions57}, \cite{Trimeche} it was established
that in the case $q\in C^{1}[-b,b]$ the Volterra-type integral operator %
\eqref{Tmain} is a transmutation in the sense of Definition \ref{DefTransmut}
on the space $C^{2}[-b,b]$ if and only if the integral kernel $\mathbf{K}%
(x,t)$ satisfies the Goursat problem \eqref{GoursatKh1}, \eqref{GoursatKh2}.
\end{remark}

The following proposition shows that to be able to construct
transmutation operators $\mathbf{T}_{h}$ or their integral kernels $\mathbf{K%
}_{h}$ for arbitrary values of the parameter $h$ it is sufficient to know the
transmutation operator $\mathbf{T}_{h_{1}}$ or its integral kernel $\mathbf{K%
}_{h_{1}}$ for some particular parameter $h_{1}$.

\begin{proposition}[\protect\cite{CKT}, \protect\cite{KT Obzor}]
\label{PropChangeOfH} The operators $\mathbf{T}_{h_{1}}$ and $\mathbf{T}%
_{h_{2}}$ are related by the equality
\begin{equation*}
\mathbf{T}_{h_{2}}u=\mathbf{T}_{h_{1}}\bigg[u(x)+\frac{h_{2}-h_{1}}{2}%
\int_{-x}^{x}u(t)\,dt\bigg]
\end{equation*}%
valid for any $u\in C[-b,b]$.

The corresponding integral kernels $\mathbf{K}(x,t;h_{1})$ and $\mathbf{K}%
(x,t;h_{2})$ are related as follows
\begin{equation*}
\mathbf{K}(x,t;h_{2})=\frac{h_{2}-h_{1}}{2}+\mathbf{K}(x,t;h_{1})+\frac{%
h_{2}-h_{1}}{2}\int_{t}^{x}\big(\mathbf{K}(x,s;h_{1})-\mathbf{K}(x,-s;h_{1})%
\big)\,ds.
\end{equation*}
\end{proposition}

The following theorem states that the operators $\mathbf{T}_{h}$ are indeed
transmutations in the sense of Definition \ref{DefTransmut}.

\begin{theorem}[\cite{KT Transmut}]
\label{Th Transmutation} Let $q\in C[-b,b]$. Then the
operator $\mathbf{T}_{h}$ defined by \eqref{Tmain} satisfies the equality
\begin{equation*}
\left( -\frac{d^{2}}{dx^{2}}+q(x)\right) \mathbf{T}_{h}[u]=\mathbf{T}_{h}%
\left[ -\frac{d^{2}}{dx^{2}}(u)\right]  
\end{equation*}
for any $u\in C^{2}[-b,b]$.
\end{theorem}

\begin{remark}
\label{RemTh}$\mathbf{T}_{h}$ maps a solution $v$ of the equation $v^{\prime
\prime }+\omega ^{2}v=0$, where $\omega $ is a complex number, into a
solution $u$ of the equation
\begin{equation}
u^{\prime \prime }-q(x)u+\omega ^{2}u=0  \label{SLomega2}
\end{equation}%
with the following correspondence of the initial values $
u(0)=v(0)$, $u^{\prime }(0)=v^{\prime }(0)+hv(0)$.
\end{remark}

Following \cite{Marchenko} we introduce the notations%
\begin{equation*}
K_{c}(x,t;h)=\mathbf{K}(x,t;h)+\mathbf{K}(x,-t;h)
\end{equation*}
where $h$ is a complex number, and $K_{s}(x,t;\infty )=\mathbf{K}(x,t;h)-\mathbf{K}(x,-t;h)$.

\begin{theorem}[\cite{Marchenko}]\label{TcTsMapsSolutions} Solutions $c(\omega ,x;h)$ and $%
s(\omega ,x;\infty )$ of equation \eqref{SLomega2} satisfying the initial
conditions
\begin{gather}
c(\omega ,0;h)=1,\qquad c_{x}^{\prime }(\omega ,0;h)=h  \label{ICcos}\\
s(\omega ,0;\infty )=0,\qquad s_{x}^{\prime }(\omega ,0;\infty )=1
\label{ICsin}
\end{gather}%
can be represented in the form
\begin{equation}
c(\omega ,x;h)=\cos \omega x+\int_{0}^{x}K_{c}(x,t;h)\cos \omega t\,dt
\label{c cos}
\end{equation}%
and
\begin{equation}
s(\omega ,x;\infty )=\frac{\sin \omega x}{\omega }+\int_{0}^{x}K_{s}(x,t;%
\infty )\frac{\sin \omega t}{\omega }\,dt.  \label{s sin}
\end{equation}
\end{theorem}

Denote by
\begin{equation*}
T_{c}u(x)=u(x)+\int_{0}^{x}K_{c}(x,t;h)u(t)dt  
\end{equation*}
and%
\begin{equation*}
T_{s}u(x)=u(x)+\int_{0}^{x}K_{s}(x,t;\infty )u(t)dt  
\end{equation*}
the corresponding integral operators.

\section{Recursive integrals and generalized wave polynomials}\label{Section3}

Let $f\in C[a,b]$ be a complex valued function and $f(x)\neq 0$ for any $%
x\in \lbrack a,b]$. The interval $(a,b)$ is assumed being finite. Let us
consider two sequences of recursive integrals%
\begin{equation}
X^{(0)}(x)\equiv 1,\qquad X^{(n)}(x)=n\int_{x_{0}}^{x}X^{(n-1)}(s)\left(
f^{2}(s)\right) ^{(-1)^{n}}\,\mathrm{d}s, \qquad x_{0}\in \lbrack a,b],\quad n=1,2,\ldots  \label{Xn}
\end{equation}%
and
\begin{equation}
\widetilde{X}^{(0)}\equiv 1,\qquad \widetilde{X}^{(n)}(x)=n\int_{x_{0}}^{x}%
\widetilde{X}^{(n-1)}(s)\left( f^{2}(s)\right) ^{(-1)^{n-1}}\,\mathrm{d}s, \qquad
x_{0}\in \lbrack a,b],\quad n=1,2,\ldots .  \label{Xtiln}
\end{equation}

Define two families of functions $\left\{ \varphi _{k}\right\}
_{k=0}^{\infty }$ and $\left\{ \psi _{k}\right\} _{k=0}^{\infty }$
constructed according to the rules
\begin{equation}
\varphi _{k}(x)=%
\begin{cases}
f(x)X^{(k)}(x), & k\text{\ odd}, \\
f(x)\widetilde{X}^{(k)}(x), & k\text{\ even},%
\end{cases}
\label{phik}
\end{equation}%
and
\begin{equation}
\psi_{k}(x)=%
\begin{cases}
\dfrac{\widetilde{X}^{(k)}(x)}{f(x)}, & k\text{\ odd,} \\
\dfrac{X^{(k)}(x)}{f(x)}, & k\text{\ even}.%
\end{cases}
\label{psik}
\end{equation}

The following result obtained in \cite{KrCV08} (for additional details and
simpler proof see \cite{APFT} and \cite{KrPorter2010}) establishes the
relation of the system of functions $\left\{ \varphi _{k}\right\}
_{k=0}^{\infty }$ and $\left\{ \psi _{k}\right\} _{k=0}^{\infty }$ to the
Sturm-Liouville equation.

\begin{theorem}
\label{ThGenSolSturmLiouville} Let $q$ be a continuous complex valued
function of an independent real variable $x\in \lbrack a,b]$ and $\lambda $
be an arbitrary complex number. Suppose there exists a solution $f$ of the
equation
\begin{equation}
f^{\prime \prime }-qf=0  \label{SLhom}
\end{equation}%
on $(a,b)$ such that $f\in C^{2}(a,b)\cap C^{1}[a,b]$ and $f(x)\neq 0$\ for
any $x\in \lbrack a,b]$. Then the general solution $y\in C^{2}(a,b)\cap
C^{1}[a,b]$ of the equation
\begin{equation}
y^{\prime \prime }-qy=\lambda y  \label{SLlambda}
\end{equation}%
on $(a,b)$ has the form $y=c_{1}y_{1}+c_{2}y_{2}$
where $c_{1}$ and $c_{2}$ are arbitrary complex constants,
\begin{equation}
y_{1}=\sum_{k=0}^{\infty }\frac{\lambda ^{k}}{(2k)!}\varphi _{2k}\qquad
\text{and}\qquad y_{2}=\sum_{k=0}^{\infty }\frac{\lambda ^{k}}{(2k+1)!}%
\varphi _{2k+1}  \label{u1u2}
\end{equation}%
and both series converge uniformly on $[a,b]$ together with the series of
the first derivatives which have the form%
\begin{multline}
y_{1}^{\prime }=f^{\prime }+\sum_{k=1}^{\infty }\frac{\lambda ^{k}}{(2k)!}%
\left( \frac{f^{\prime }}{f}\varphi _{2k}+2k\,\psi _{2k-1}\right) \qquad
\text{and}  \label{du1du2} \\
y_{2}^{\prime }=\sum_{k=0}^{\infty }\frac{\lambda ^{k}}{(2k+1)!}\left( \frac{%
f^{\prime }}{f}\varphi _{2k+1}+\left( 2k+1\right) \psi _{2k}\right) .
\end{multline}%
The series of the second derivatives converge uniformly on any segment $%
[a_{1},b_{1}]\subset (a,b)$.
\end{theorem}

Representations \eqref{u1u2} and \eqref{du1du2}, also known as the SPPS
method (Spectral Parameter Power Series), present an efficient and highly
competitive technique for solving a variety of spectral and scattering
problems related to Sturm-Liouville equations. The first work implementing
Theorem \ref{ThGenSolSturmLiouville} for numerical solution was \cite%
{KrPorter2010} and later on the SPPS method was used in a number of
publications (see \cite{CKOR}, \cite{ErbeMertPeterson2012}, \cite{KKB}, \cite%
{KKRosu}, \cite{KiraRosu2010}, \cite{KT Obzor} and references therein).

\begin{remark}
\label{RemInitialValues}It is easy to see that by definition the solutions $%
y_{1}$ and $y_{2}$ from (\ref{u1u2}) satisfy the following initial
conditions
\begin{align*}
y_{1}(x_{0})& =f(x_{0}), & y_{1}^{\prime }(x_{0})& =f^{\prime }(x_{0}), \\
y_{2}(x_{0})& =0, & y_{2}^{\prime }(x_{0})& =1/f(x_{0}).
\end{align*}
\end{remark}

\begin{remark}
\label{RemarkNonVanish} It is worth mentioning that in the regular case the
existence and construction of the required $f$ presents no difficulty.
Indeed, let $q$ be real valued and continuous on $[a,b]$. Then (\ref{SLhom})
possesses two linearly independent real-valued solutions $f_{1}$ and $f_{2} $ whose zeros alternate. Thus, one may choose $f=f_{1}+if_{2}$.
Moreover, for the construction of $f_{1}$ and $f_{2}$ in fact the same SPPS
method may be used \cite{KrPorter2010}. In the case of complex-valued coefficients the existence of a non-vanishing solution was shown in \cite[Remark 5]{KrPorter2010}.
\end{remark}

In what follows we choose $x_{0}=0$.

\begin{theorem}[\protect\cite{CKT}, \protect\cite{KrT2012}]
\label{Th Transmute}Let $q$ be a continuous complex valued function of an
independent real variable $x\in \lbrack -b,b]$ for which there exists a
particular solution $f$ of \eqref{SLhom} such that $f\in C^{2}[-b,b]$, $%
f\neq 0$ on $[-b,b]$ and normalized as $f(0)=1$. Denote $h:=f^{\prime
}(0)\in \mathbb{C}$. Suppose $\mathbf{T}_{h}$ is the operator defined by %
\eqref{Tmain} and $\varphi _{k}$, $k\in \mathbb{N}_{0}$ are functions
defined by \eqref{phik}. Then
\begin{equation}
\mathbf{T}_{h}x^{k}=\varphi _{k}(x)\qquad \text{for any}\ k\in \mathbb{N}%
_{0}.  \label{mapping powers 1}
\end{equation}
\end{theorem}

\begin{remark}
The mapping property \eqref{mapping powers 1} of the transmutation operator
allows one to see that the SPPS representations \eqref{u1u2} from Theorem %
\ref{ThGenSolSturmLiouville} are nothing but the images of Taylor expansions
of the functions $\cosh \sqrt{\lambda }x$ and $\frac{1}{\sqrt{\lambda }}%
\sinh \sqrt{\lambda }x$ under the action of $\mathbf{T}_{h}$. Moreover,
equality \eqref{mapping powers 1} is behind a new method for solving
Sturm-Liouville problems proposed in \cite{KT MMET 2012} and based on the
use of Tchebyshev polynomials for approximating trigonometric functions,
combined with \eqref{mapping powers 1}.
\end{remark}

In what follows we assume that $f\in C^{2}[-b,b]$, $f\neq 0$ on $[-b,b]$, $%
f(0)=1$ and denote $h:=f^{\prime }(0)\in \mathbb{C}$. Any such function is
associated with an operator $\mathbf{T}_{h}$. For convenience, from now on
we will write $T_{f}$ instead of $\mathbf{T}_{h}$ and the integral kernel of
$T_{f}$ will be denoted by $\mathbf{K}_{f}$. The kernel $K_{c}(x,t;h)$ will
be denoted by $C_{f}(x,t)$ and the kernel $K_{s}(x,t;\infty )$ by $%
S_{f}(x,t) $.

Notice that
\begin{equation}
C_{f}(x,t)=\mathbf{K}_{f}(x,t)+\mathbf{K}_{f}(x,-t)  \label{Cf}
\end{equation}%
and
\begin{equation}
S_{f}(x,t)=\mathbf{K}_{f}(x,t)-\mathbf{K}_{f}(x,-t).  \label{Sf}
\end{equation}

The following functions introduced in \cite{KKTT}%
\begin{equation}
u_{0}=\varphi _{0}(x),\quad u_{2m-1}(x,t)=\sum_{\text{even }k=0}^{m}\binom{m%
}{k}\varphi _{m-k}(x)t^{k},\quad u_{2m}(x,t)=\sum_{\text{odd }k=1}^{m}\binom{%
m}{k}\varphi _{m-k}(x)t^{k},  \label{um}
\end{equation}%
are called generalized wave polynomials (the wave polynomials are introduced
below, in Example \ref{Ex Wave polynomials}). The following parity relations
hold for the generalized wave polynomials.
\begin{equation}
u_{0}(x,-t)=u_{0}(x,t),\qquad u_{2n-1}(x,-t)=u_{2n-1}(x,t),\qquad
u_{2n}(x,-t)=-u_{2n}(x,t).  \label{umParity}
\end{equation}

For the values of the generalized wave polynomials on the characteristics $%
x=t$ and $x=-t$ we introduce the additional notations%
\begin{equation}
\mathbf{c}_{m}(x)=u_{2m-1}(x,x)=\sum_{\text{even }k=0}^{m}\binom{m}{k}%
x^{k}\varphi _{m-k}(x),\quad m=1,2,\ldots \text{ and }\mathbf{c}%
_{0}(x)=u_{0}(x,x)=f(x),  \label{cm}
\end{equation}%
\begin{equation}
\mathbf{s}_{m}(x)=u_{2m}(x,x)=\sum_{\text{odd }k=1}^{m}\binom{m}{k}%
x^{k}\varphi _{m-k}(x),\quad m=1,2,\ldots .  \label{sm}
\end{equation}%
As we show in Sections \ref{SectApprox} and \ref{Section6} the systems of functions $\left\{
\mathbf{c}_{m}\right\} _{m=0}^{\infty }$, $\left\{ \mathbf{s}_{m}\right\}
_{m=1}^{\infty }$ and $\left\{ u_{m}\right\} _{m=0}^{\infty }$ play a
crucial role in the construction of the transmutation kernels and hence of
the solutions of equation (\ref{SLlambda}).

\begin{example}
\label{Ex Wave polynomials}In a special case when $f\equiv 1$ we obtain that
$\varphi _{k}(x)=x^{k}$, $k\in \mathbb{N}_{0}$ and $u_{k}(x,t)=p_{k}(x,t)$
where $p_{k}$ are wave polynomials \cite[Proposition 1]{KKTT} defined by the
equalities
\begin{equation*}
p_{0}(x,t)=1,\quad p_{2m-1}(x,t)=\mathcal{R}\bigl((x+\mathbf{j}t)^{m}\bigr)%
,\quad p_{2m}(x,t)=\mathcal{I}\bigl((x+\mathbf{j}t)^{m}\bigr),\ m\geq 1.
\end{equation*}%
Here $\mathbf{j}$ is a hyperbolic imaginary unit (see, e.g., \cite%
{Lavrentyev and Shabat}, \cite{MotterRosa}, \cite{Sobczyk} and \cite{APFT}):
$\mathbf{j}^{2}=1$ and $\mathbf{j}\neq \pm 1$, $\mathcal{R}$ and $\mathcal{I}
$ are the real and the imaginary parts respectively of a corresponding
hyperbolic number. The wave polynomials may also be written as follows
\begin{equation*}
p_{0}(x,t)=1,\quad p_{2m-1}(x,t)=\sum_{\mathrm{even}\text{ }k=0}^{m}\binom{m%
}{k}x^{m-k}t^{k},\quad p_{2m}(x,t)=\sum_{\mathrm{odd}\text{ }k=1}^{m}\binom{m%
}{k}x^{m-k}t^{k}.  
\end{equation*}
We see that in this special case $\mathbf{c}_{m}(x)=\mathbf{s}%
_{m}(x)=2^{m-1}x^{m}$, $m=1,2,\ldots $.
\end{example}

\section{Goursat-to-Goursat transmutation operators and completeness of the systems $\{\mathbf{c}_n\}$ and $\{\mathbf{s}_n\}$}\label{Section4}

By $\overline{\mathbf{S}}$ we denote a closed square with a diagonal joining
the endpoints $(b,b)$ and $(-b,-b)$. \ Let $\square :=\partial
_{x}^{2}-\partial _{t}^{2}$ and the functions $\widetilde{u}$ and $u$ be
solutions of the equations $\square \widetilde{u}=0$ and $\left( \square
-q(x)\right) u=0$ in $\overline{\mathbf{S}}$, respectively such that $u=T_{f}%
\widetilde{u}$. Following \cite{KT Transmut} we consider the operator $T_{G}$
mapping the Goursat data corresponding to $\widetilde{u}$ into the Goursat
data corresponding to $u$,
\begin{equation*}
T_{G}:\quad \binom{\widetilde{u}(x,x)}{\widetilde{u}(x,-x)}\quad \longmapsto
\quad \binom{u(x,x)}{u(x,-x)}.
\end{equation*}%
The operator $T_{G}$ is well defined on the linear space $V$ of vector
functions $\begin{pmatrix}
\varphi \\
\psi
\end{pmatrix}$ from $C^{1}[-b,b]\times C^{1}[-b,b]$ such that $\varphi (0)=\psi
(0)$, equipped with the maximum norm.

\begin{proposition}[\cite{KT Transmut}]\label{Prop Goursat to Goursat} The operator $T_{G}$
together with $T_{G}^{-1}$ are bounded on $V$ and the action of the operator
$T_{G}$ is defined by the following relation
\begin{equation*}
T_{G}
\begin{pmatrix}
\varphi (x) \\
\psi (x)
\end{pmatrix}=
\begin{pmatrix}
\varphi (x)+2\int_{0}^{x}\mathbf{K}_{f}(x,2t-x)\varphi (t)dt+\psi
(0)+2\int_{-x}^{0}\mathbf{K}_{f}(x,2t+x)\psi (t)dt-\varphi (0)f(x) \\
\varphi (0)+2\int_{-x}^{0}\mathbf{K}_{f}(x,2t+x)\varphi (t)dt+\psi
(x)+2\int_{0}^{x}\mathbf{K}_{f}(x,2t-x)\psi (t)dt-\varphi (0)f(x)%
\end{pmatrix}.
\end{equation*}%
In particular,
\begin{align*}
T_{G}:\ 2^{n-1}
\begin{pmatrix}
x^{n}\\
x^{n}\end{pmatrix}& \longmapsto  \begin{pmatrix}
\mathbf{c}_{n}(x)\\
\mathbf{c}_{n}(x)\end{pmatrix},\\  
T_{G}:\ 2^{n-1}\binom{x^{n}}{-x^{n}}& \longmapsto \begin{pmatrix}
\mathbf{s}_{n}(x)\\
-\mathbf{s}_{n}(x)\end{pmatrix}  
\end{align*}
for $n\in \mathbb{N}$ and $T_{G}: \left(\begin{smallmatrix}
1\\
1\end{smallmatrix}\right) \longmapsto \left(\begin{smallmatrix}
\mathbf{c}_{0}(x)\\
\mathbf{c}_{0}(x)\end{smallmatrix}\right)$.
\end{proposition}

Using the properties of the operator $T_{G}$ the following proposition was
obtained in \cite{KT Transmut}.

\begin{proposition}[\cite{KT Transmut}]
\label{Prop Unif Conv u} Let $u$ be a regular solution of
the equation
\begin{equation}
\left( \square -q(x)\right) u=0  \label{KG}
\end{equation}%
in $\overline{\mathbf{S}}$ such that its Goursat data admit the following
series expansions%
\begin{equation*}
\frac{1}{2}\left( u(x,x)+u(x,-x)\right) =\sum_{n=0}^{\infty }a_{n}\mathbf{c}%
_{n}(x),  
\end{equation*}
and
\begin{equation*}
\frac{1}{2}\left( u(x,x)-u(x,-x)\right) =\sum_{n=1}^{\infty }b_{n}\mathbf{s}%
_{n}(x),  
\end{equation*}
both uniformly convergent on $[-b,b]$. Then for any $(x,t)\in \overline{%
\mathbf{S}}$,
\begin{equation*}
u(x,t)=a_{0}u_{0}(x,t)+\sum_{n=1}^{\infty }\left(
a_{n}u_{2n-1}(x,t)+b_{n}u_{2n}(x,t)\right)  
\end{equation*}
and the series converges uniformly in $\overline{\mathbf{S}}$.
\end{proposition}

It is not difficult to prove the linear independence as well as the
completeness of the families of functions $\left\{ \mathbf{c}_{n}\right\}
_{n=0}^{\infty }$ and $\left\{ \mathbf{s}_{n}\right\} _{n=1}^{\infty }$ in
appropriate functional spaces. For this together with the operator $T_{G}$
it is convenient to consider the following its modification,%
\begin{equation}
G:=UT_{G}U  \label{opG}
\end{equation}%
where
\begin{equation*}
U:=\frac{1}{\sqrt{2}}
\begin{pmatrix}
1 & 1 \\
1 & -1
\end{pmatrix}.
\end{equation*}%
Notice that $U^{2}=
\begin{pmatrix}
1 & 0 \\
0 & 1
\end{pmatrix}$ and hence $U=U^{-1}$.

Let us consider the operator $G$ on the space $C^{1}[-b,b]\times
C_{0}^{1}[-b,b]$ where $C_{0}^{n}[-b,b]$ denotes a subspace of $C^{n}[-b,b]$
consisting of functions vanishing in the origin. The operator $G$ transforms
the half-sum and the half-difference of Goursat data for solutions of the
wave equation into their counterpart for solutions of (\ref{KG}),
\begin{equation*}
G:\ \frac{1}{2}\begin{pmatrix}
\widetilde{u}(x,x)+\widetilde{u}(x,-x)\\
\widetilde{u}(x,x)-\widetilde{u}(x,-x)
\end{pmatrix} \longmapsto \frac{1}{2}\begin{pmatrix}
u(x,x)+u(x,-x)\\
u(x,x)-u(x,-x)\end{pmatrix}.
\end{equation*}%
It is bounded together with its inverse. Note that from Proposition \ref%
{Prop Goursat to Goursat} we have%
\begin{align}
G:\ \begin{pmatrix}1\\0\end{pmatrix} &\longmapsto  \begin{pmatrix}\mathbf{c}_{0}(x)\\0\end{pmatrix},
\label{map c0}\\
G:\ 2^{n-1}\begin{pmatrix}x^{n}\\
0\end{pmatrix} &\longmapsto  \begin{pmatrix}\mathbf{c}_{n}(x)\\
0\end{pmatrix}  \label{map cn}
\end{align}%
and
\begin{equation}
G:\ 2^{n-1}\begin{pmatrix}0\\
x^{n}\end{pmatrix} \longmapsto  \begin{pmatrix}0\\
\mathbf{s}_{n}(x)\end{pmatrix}.  \label{map sn}
\end{equation}

\begin{proposition}\mbox{}
\begin{enumerate}
\item[1)] The operator $G$ defined by \eqref{opG} on $C^{1}[-b,b]\times
C_{0}^{1}[-b,b]$ admits the following representation
\begin{equation*}
G\begin{pmatrix}\eta (x)\\
\xi (x)\end{pmatrix}=\begin{pmatrix}G_{+}\left[ \eta (x)-\frac{\eta (0)
}{2}\right] +\frac{\eta (0)}{2}\\
G_{-}\left[ \xi (x)\right] \end{pmatrix}  
\end{equation*}
where $G_{+}$ and $G_{-}$ have the form%
\begin{equation}
G_{\pm }\eta (x)=\eta (x)+\int_{-x}^{x}\mathcal{K}_{\pm }(x,t)\eta (t)dt
\label{Gform1}
\end{equation}%
with the kernels given by the equalities%
\begin{equation*}
\mathcal{K}_{\pm }(x,t):=\begin{cases}
\pm 2\mathbf{K}_{f}(x,2t+x), & -x\leq t<0, \\
2\mathbf{K}_{f}(x,2t-x) & 0\leq t\leq x.
\end{cases}
\end{equation*}%
The operators $G_{+}$ and $G_{-}$ can also be written in the form
\begin{equation}
G_{\pm }\eta (x)=\eta (x)+\int_{-x}^{x}\mathbf{K}_{f}(x,t)\left( \eta \left(\frac{%
t+x}{2}\right)\pm \eta \left(\frac{t-x}{2}\right)\right) dt.  \label{Gform2}
\end{equation}
\item[2)] Both operators $G_{+}$ and $G_{-}$ preserve the value of the function in
the origin and $G_{+}:C^{1}[-b,b]\rightarrow C^{1}[-b,b]$, $%
G_{-}:C_{0}^{1}[-b,b]\rightarrow C_{0}^{1}[-b,b]$.
\item[3)] There exist the inverse operators $G_{+}^{-1}$ and $G_{-}^{-1}$ defined
on $C^{1}[-b,b]$ and $C_{0}^{1}[-b,b]$, respectively, and the inverse
operator for $G$ admits the representation
\begin{equation}
G^{-1}\begin{pmatrix}\eta (x)\\
\xi (x)\end{pmatrix}=\begin{pmatrix}G_{+}^{-1}\left[ \eta (x)-\frac{\eta
(0)}{2}\right] +\frac{\eta (0)}{2}\\
G_{-}^{-1}\left[ \xi (x)\right] \end{pmatrix}.
\label{G-1}
\end{equation}
\end{enumerate}
\end{proposition}

\begin{proof}
1) Let us observe that (\ref{Gform2}) is obtained from (\ref{Gform1}) by a
simple change of variables. From Proposition \ref{Prop Goursat to Goursat}
we obtain
\begin{equation*}
G\binom{\eta (x)}{\xi (x)}=\binom{G_{+}\left[ \eta (x)\right] +\eta
(0)(1-f(x))}{G_{-}\left[ \xi (x)\right] }
\end{equation*}%
and observe that if $c$ is a complex constant then $G_{+}\left[ c\right]
=2c(f(x)-1)+c$. Indeed, from (\ref{Gform2}) we have that
\begin{equation}
G_{+}\left[ c\right] =c+2c\int_{-x}^{x}\mathbf{K}_{f}(x,t)dt=2cT_{f}\left[ 1%
\right] -c=2cf(x)-c  \label{G+c}
\end{equation}%
where we used (\ref{mapping powers 1}). Hence $G_{+}\left[ \frac{\eta (0)}{2}%
\right] =\eta (0)(f(x)-1)+\frac{\eta (0)}{2}$.

The proof of 2) follows directly from (\ref{Gform2}).

3) The existence of the inverse operators $G_{+}^{-1}$ and $G_{-}^{-1}$
follows from (\ref{Gform1}), the fact that the kernels $\mathcal{K}_{\pm }$
are bounded measurable functions and a well known result on the
invertibility of Volterra operators with such kernels (see \cite[Ch. 9,
Subsect. 4.5]{KolmFom}). The representation (\ref{G-1}) can be verified
directly using 2) and (\ref{G+c}).
\end{proof}

\begin{remark}
Observe that $G$ (and hence $G^{-1}$) results to be a diagonal operator,
\begin{equation*}
G=\begin{pmatrix}
G_{1} & 0 \\
0 & G_{2}%
\end{pmatrix}
\end{equation*}%
with $G_{1}=G_{+}(I-\frac{1}{2}\delta )+\frac{1}{2}\delta $ and $G_{2}=G_{-}$
where $I$ is the identity operator and $\delta $ is the functional acting as
follows $\delta \left[ \eta (x)\right] =\eta (0)$.

From (\ref{map c0}), (\ref{map cn}) and (\ref{map sn}) we have $G_{1}\left[ 1%
\right] =\mathbf{c}_{0}(x)=f(x)$, $2^{n-1}G_{1}\left[ x^{n}\right] =\mathbf{c%
}_{n}(x)$ and $2^{n-1}G_{2}\left[ x^{n}\right] =\mathbf{s}_{n}(x)$ for $n\in
\mathbb{N}$.
\end{remark}

\begin{corollary}
The systems of functions $\left\{ \mathbf{c}_{n}\right\} _{n=0}^{\infty }$
and $\left\{ \mathbf{s}_{n}\right\} _{n=1}^{\infty }$ are linearly
independent and complete in $C^{1}[-b,b]\ $and $C_{0}^{1}[-b,b]$
respectively.
\end{corollary}

\begin{proof}
To prove the linear independence of functions $\mathbf{c}_{n}$ consider a
nontrivial linear combination $\mathbf{c}:=a_{0}\mathbf{c}_{0}+\ldots +a_{N}%
\mathbf{c}_{N}$ and suppose that $\mathbf{c}\equiv 0$ on $[-b,b]$. We have
then $G_{1}^{-1}\left[ \mathbf{c}(x)\right] =a_{0}+\ldots +2^{N-1}a_{N}x^{N}$
and hence $a_{0}+\ldots +2^{N-1}a_{N}x^{N}\equiv 0$ on $[-b,b]$ which is a
contradiction. The linear independence of functions $\mathbf{s}_{n}$ is
proved analogously.

The completeness of $\left\{ \mathbf{c}_{n}\right\} _{n=0}^{\infty }$ in $%
C^{1}[-b,b]\ $follows from the completeness of the powers $\left\{
x^{n}\right\} _{n=0}^{\infty }$ in $C^{1}[-b,b]\ $and properties of $G_{1}$.
Analogously, the completeness of $\left\{ \mathbf{s}_{n}\right\}
_{n=1}^{\infty }$ in $C_{0}^{1}[-b,b]\ $follows from the completeness of the
powers $\left\{ x^{n}\right\} _{n=1}^{\infty }$ in $C_{0}^{1}[-b,b]\ $and
properties of $G_{2}$.
\end{proof}

\section{Approximate construction of integral kernels\label{SectApprox}}

\begin{theorem}
\label{Th Kapprox}Let the complex numbers $a_{0},\ldots ,a_{N}$ and $%
b_{1},\ldots ,b_{N}$ be such that%
\begin{equation}\label{KxxErr}
\left\vert \frac{h}{2}+\frac{1}{4}\int_{0}^{x}q(s)ds-\sum_{n=0}^{N}a_{n}%
\mathbf{c}_{n}(x)\right\vert <\varepsilon _{1}
\end{equation}%
and
\begin{equation}\label{KxmxErr}
\left\vert \frac{1}{4}\int_{0}^{x}q(s)ds-\sum_{n=1}^{N}b_{n}\mathbf{s}%
_{n}(x)\right\vert <\varepsilon _{2}
\end{equation}%
for every $x\in \lbrack -b,b]$. Then the kernel $\mathbf{K}_{f}(x,t)$ is
approximated by the function%
\begin{equation}
K_{f,N}(x,t)=a_{0}u_{0}(x,t)+\sum_{n=1}^{N}a_{n}u_{2n-1}(x,t)+%
\sum_{n=1}^{N}b_{n}u_{2n}(x,t)  \label{K(x,t)}
\end{equation}%
in such a way that for every $(x,t)\in \overline{\mathbf{S}}$ the inequality
holds
\begin{equation}
\bigl|\mathbf{K}_{f}(x,t)-K_{f,N}(x,t)\bigr|\leq 3\Vert T_{f}\Vert \cdot
\Vert T_{G}^{-1}\Vert (\varepsilon _{1}+\varepsilon _{2}).  \label{estim}
\end{equation}
\end{theorem}

\begin{proof}
Notice that $g_{1}(x):=\frac{h}{2}+\frac{1}{4}\int_{0}^{x}q(s)ds=\frac{1}{2}%
\left( \mathbf{K}_{f}(x,x)+\mathbf{K}_{f}(x,-x)\right) $ and $g_{2}(x):=%
\frac{1}{4}\int_{0}^{x}q(s)ds=\frac{1}{2}\left( \mathbf{K}_{f}(x,x)-\mathbf{K%
}_{f}(x,-x)\right) $ and hence the theorem establishes that if the half-sum
and the half-difference of the Goursat data corresponding to $\mathbf{K}%
_{f}(x,t)$ are approximated by linear combinations of the functions $\mathbf{%
c}_{n}$ and $\mathbf{s}_{n}$ respectively, the function (\ref{K(x,t)})
approximates uniformly the kernel $\mathbf{K}_{f}(x,t)$. Indeed, consider
the functions $\widetilde{\mathbf{K}}_{f}=T_{f}^{-1}\mathbf{K}_{f}$ and $%
\widetilde{K}_{f,N}=T_{f}^{-1}K_{f,N}$. Then by the definition of the
Goursat-to-Goursat transmutation operator
\begin{equation*}
\begin{pmatrix}
\widetilde{\mathbf{K}}_{f}(x,x) \\
\widetilde{\mathbf{K}}_{f}(x,-x)%
\end{pmatrix}%
=T_{G}^{-1}%
\begin{pmatrix}
\mathbf{K}_{f}(x,x) \\
\mathbf{K}_{f}(x,-x)%
\end{pmatrix}%
\quad \text{and}\quad
\begin{pmatrix}
\widetilde{K}_{f,N}(x,x) \\
\widetilde{K}_{f,N}(x,-x)%
\end{pmatrix}%
=T_{G}^{-1}%
\begin{pmatrix}
K_{f,N}(x,x) \\
K_{f,N}(x,-x)%
\end{pmatrix}%
,
\end{equation*}%
hence due to the boundedness of the operator $T_{G}^{-1}$
\begin{multline*}
\max \biggl(\max_{x\in \lbrack -b,b]}\bigl|\widetilde{\mathbf{K}}_{f}(x,x)-%
\widetilde{K}_{f,N}(x,x)\bigr|,\max_{x\in \lbrack -b,b]}\bigl|\widetilde{%
\mathbf{K}}_{f}(x,-x)-\widetilde{K}_{f,N}(x,-x)\bigr|\biggr) \\
\leq \Vert T_{G}^{-1}\Vert \max \biggl(\max_{x\in \lbrack -b,b]}\bigl|%
\mathbf{K}_{f}(x,x)-K_{f,N}(x,x)\bigr|,\max_{x\in \lbrack -b,b]}\bigl|%
\mathbf{K}_{f}(x,-x)-K_{f,N}(x,-x)\bigr|\biggr) \displaybreak[2]\\
\leq \Vert T_{G}^{-1}\Vert \max_{x\in \lbrack -b,b]}\biggl(\biggl|\frac{1}{2}%
\bigl(\mathbf{K}_{f}(x,x)+\mathbf{K}_{f}(x,-x)\bigr)-\frac{1}{2}\bigl(%
K_{f,N}(x,x)+K_{f,N}(x,-x)\bigr)\biggr| \\
+\biggl|\frac{1}{2}\bigl(\mathbf{K}_{f}(x,x)-\mathbf{K}_{f}(x,-x)\bigr)-%
\frac{1}{2}\bigl(K_{f,N}(x,x)-K_{f,N}(x,-x)\bigr)\biggr|\biggr)<\Vert
T_{G}^{-1}\Vert (\varepsilon _{1}+\varepsilon _{2}),
\end{multline*}%
where we have used the equalities $\frac{1}{2}\bigl(%
K_{f,N}(x,x)+K_{f,N}(x,-x)\bigr)=a_{0}u_{0}(x,x)+%
\sum_{n=1}^{N}a_{n}u_{2n-1}(x,x)$ and $\frac{1}{2}\bigl(%
K_{f,N}(x,x)-K_{f,N}(x,-x)\bigr)=\sum_{n=1}^{N}b_{n}u_{2n}(x,x)$. We obtain
from the proof of \cite[Theorem 3]{KKTT} that for every $(x,t)\in \overline{%
\mathbf{S}}$
\begin{equation*}
\bigl|\widetilde{\mathbf{K}}_{f}(x,t)-\widetilde{K}_{f,N}(x,t)\bigr|\leq
3\Vert T_{G}^{-1}\Vert (\varepsilon _{1}+\varepsilon _{2}),
\end{equation*}%
hence for every $(x,t)\in \overline{\mathbf{S}}$
\begin{equation*}
\bigl|\mathbf{K}_{f}(x,t)-K_{f,N}(x,t)\bigr|\leq 3\Vert T_{f}\Vert \cdot
\Vert T_{G}^{-1}\Vert (\varepsilon _{1}+\varepsilon _{2}).\qedhere
\end{equation*}
\end{proof}

\begin{remark}
It is not difficult to make the estimate (\ref{estim}) more explicit. For
example, let the function $\mathbf{H}(u,v)$ be a solution of the Goursat
problem for equation (\ref{GoursatTh1}) with the conditions $\mathbf{H}%
(u,0)=\varepsilon _{1}(u)$ and $\mathbf{H}(0,v)=\varepsilon _{2}(v)$, $\varepsilon_1(0)=\varepsilon_2(0)=\varepsilon_0$, where $\varepsilon _{1}(u)=\frac{h}{2}+\frac{1}{2}\int_{0}^{u}q(s)ds-
\sum_{n=0}^{N}a_{n}\mathbf{c}_{n}(u)-\sum_{n=1}^{N}b_{n}\mathbf{s}_{n}(u)$
and $\varepsilon _{2}(v)=\frac{h}{2}-\sum_{n=0}^{N}a_{n}\mathbf{c}%
_{n}(v)+\sum_{n=1}^{N}b_{n}\mathbf{s}_{n}(v)$ are the differences between the
exact and the approximate transmutation kernels on the characteristics.
Then similarly to \cite[Subsect. 15.1]{Vladimirov} one can see that the Goursat problem is equivalent to the integral equation $\mathbf{H}(u,v)=\int_0^u\int_0^v q(u'+v')\mathbf{H}(u',v')\,du'\,dv'+\varepsilon_1(u)+\varepsilon_2(v)-\varepsilon_0$. Applying the successive approximations technique one obtains for $\mathbf{H}$ the following estimate $\max|\mathbf{H}(u,v)|\le m I_0(2\kappa\sqrt{|uv|})$, where $I_0$ is the modified Bessel function of the first kind, $\kappa =\sqrt{\max \left\vert q\right\vert}$ and $m= \max \left\vert \varepsilon _{1}\right\vert +\max \left\vert
\varepsilon _{2}\right\vert+\max \left\vert \varepsilon
_{0}\right\vert$. Thus, $\bigl|\mathbf{K}
_{f}(x,t)-K_{f,N}(x,t)\bigr|\leq mI_0 (\kappa x)$.
\end{remark}

\begin{remark}\label{Rm HalfSegment}
Let us notice that the approximation of the transmutation kernel $\mathbf{K}%
_{f}(x,t)$ in the form (\ref{K(x,t)}) implies the following approximations
of the kernels
\begin{equation}
S_{f}(x,t)\cong S_N(x,t):=2\sum_{n=1}^{N}b_{n}u_{2n}(x,t)  \label{Sapprox}
\end{equation}%
and
\begin{equation}
C_{f}(x,t)\cong C_N(x,t):=2\left(
a_{0}u_{0}(x,t)+\sum_{n=1}^{N}a_{n}u_{2n-1}(x,t)\right) .  \label{Capprox}
\end{equation}%
This is a direct corollary of Theorem \ref{Th Kapprox}, formulas (\ref{Cf}),
(\ref{Sf}) and (\ref{umParity}). Notice that to obtain the coefficients $%
a_{n}$ and $b_{n}$ in (\ref{Sapprox}) and (\ref{Capprox}) the approximation
of $g_{1}$ and $g_{2}$ can be performed on $[0,b]$ only and hence for $q\in
C[0,b]$.
\end{remark}

The approximation of the functions $g_{1}$ and $g_{2}$ by the corresponding
combinations of the functions $\mathbf{c}_{n}$ and $\mathbf{s}_{n}$ can be
done in several ways. For example, the least squares method can be used to obtain a reasonably good approximation. Even though its not clear how to verify whether the systems of functions $\mathbf{c}_n$ and $\mathbf{s}_n$ are Tchebyshev systems, in the case when all the involved functions are real valued the Remez algorithm can be used, see \cite[Section 6]{KKTT} and references therein. Another alternative is to reformulate the
approximation problem as a linear programming problem and solve it.

\section{Approximate solution}\label{Section6}

Let us explain the special convenience of the approximations of the
transmutation kernels in terms of the generalized wave polynomials.
Consider, for example, the kernel $S_{f}(x,t)$ which by (\ref{s sin})
transforms the function $\frac{\sin \omega x}{\omega }$ into a solution of (%
\ref{SLomega2}) satisfying the initial conditions (\ref{ICsin}) (see Theorem %
\ref{TcTsMapsSolutions}). Now instead of the exact kernel $S_{f}(x,t)$ let
us substitute into (\ref{s sin}) its approximation $S_N(x,t)$. We obtain an
approximate solution
\begin{equation*}
\begin{split}
s(\omega ,x;\infty ) \cong s_{N}(\omega ,x) &=\frac{\sin \omega x}{\omega }+\int_{0}^{x}S_N(x,t)\frac{
\sin \omega t}{\omega }\,dt \\
&=\frac{\sin \omega x}{\omega }+2\int_{0}^{x}\sum_{n=1}^{N}b_{n}u_{2n}(x,t)%
\frac{\sin \omega t}{\omega }\,dt.
\end{split}
\end{equation*}
By the definition of the generalized wave polynomials we have
\begin{equation}\label{sN}
s_{N}(\omega ,x)=\frac{1}{\omega }\left( \sin \omega
x+2\sum_{n=1}^{N}b_{n}\sum_{\text{odd }k=1}^{n}\binom{n}{k}\varphi
_{n-k}(x)\int_{0}^{x}t^{k}\sin \omega t\,dt\right) .
\end{equation}
The integrals here are, of course, easily calculated explicitly. For
example, the following formula can be used \cite[2.633]{GradRizh}
\begin{equation}\label{Int_tsin}
\int_{{}}^{{}}t^{k}\sin (\omega t)dt=-\sum\limits_{j=0}^{k}j!\binom{k}{j}%
\frac{t^{k-j}}{\omega ^{j+1}}\cos \left( \omega t+\frac{j\pi }{2}\right),
\end{equation}%
or alternatively the integrals can be calculated recursively.

Analogously, the approximation of the solution $c(\omega ,x;h)$ is
calculated as follows
\begin{equation}\label{cN}
\begin{split}
c(\omega ,x;h) \cong c_{N}(\omega ,x) &=\cos \omega x+\int_{0}^{x}C_N(x,t)\cos \omega t\,dt \\
&=\cos \omega x+2\sum_{n=0}^{N}a_{n}\sum_{\text{even }k=0}^{n}\binom{n}{k}\varphi
_{n-k}(x)\int_{0}^{x}t^{k}\cos \omega t\,dt
\end{split}
\end{equation}
where the integrals can be calculated exactly using the formula \cite[2.633]%
{GradRizh}
\begin{equation}\label{Int_tcos}
\int_{{}}^{{}}t^{k}\cos (\omega t)dt=\sum\limits_{j=0}^{k}j!\binom{k}{j}%
\frac{t^{k-j}}{\omega ^{j+1}}\sin \left( \omega t+\frac{j\pi }{2}\right) .
\end{equation}%
Thus, the problem of approximate solution of equation (\ref{SLomega2}) can be reduced to the problem of approximation of the functions $g_{1}$, $g_{2}$ in terms of the functions $\mathbf{c}_{n}$ and $\mathbf{s}_{n}$ respectively.

\begin{remark}
For a real $\omega $ we have that the accuracy of the approximate solution does not
deteriorate when $\omega $ increases. Indeed, considering, e.g., $\left\vert
c(\omega ,x;h)-c_{N}(\omega ,x)\right\vert $ under the assumption $%
\left\vert C_{f}(x,t)-C_N(x,t)\right\vert \leq \varepsilon $ we have%
\begin{equation*}
\left\vert c(\omega ,x;h)-c_{N}(\omega ,x)\right\vert \leq
\int_{0}^{x}\left\vert C_{f}(x,t)-C_N(x,t)\right\vert \left\vert \cos \omega
t\,\right\vert dt\leq \varepsilon \int_{0}^{x}\left\vert \cos \omega
t\,\right\vert dt\leq \varepsilon \left\vert x\right\vert .
\end{equation*}
\end{remark}

A similar observation is true also for complex values of the parameter $%
\omega $. Indeed, for an arbitrary complex valued potential $q\in C[0,b]$
and arbitrary nondegenerate boundary conditions the asymptotic formulas for
the square roots of eigenvalues (see \cite[Chapter 1, Sect.5]{Marchenko})
tell us that all the square roots of eigenvalues are located in a strip on a
complex plane parallel to the real axis. For example, in the case of the
problem for (\ref{SLomega2}) \ with the boundary conditions $u(0)=u(\pi )=0$
one has that $\omega _{k}=k+\theta (k)$ where $\theta (k)\rightarrow 0$ when
$k\rightarrow \infty $ \cite[p. 69]{Marchenko}. Analogous asymptotic
formulas are available for all other nondegenerate boundary conditions.
Thus, solution of a nondegenerate Sturm-Liouville problem implies
consideration of solutions $c(\omega ,x;h)$ and $s(\omega ,x;\infty )$ for $%
\omega $ with $\operatorname{Im}\omega $ belonging to a finite interval. The
following statement establishes that similarly to the case $\operatorname{Im}\omega
=0$, when $\left\vert \operatorname{Im}\omega \right\vert \leq \mathrm{Const}$ the
accuracy of approximation does not depend on $\omega $.

\begin{proposition}
Let the parameter $\omega $ belong to the strip $\left\vert \operatorname{Im}\omega
\right\vert \leq C$ where $C$ is a positive number. Suppose that $\max_{%
\overline{S}}\left\vert C_{f}(x,t)-C_N(x,t)\right\vert \leq \varepsilon $.
Then
\begin{equation}
\left\vert c(\omega ,x;h)-c_{N}(\omega ,x)\right\vert \leq \varepsilon
\sinh (Cx)/C.  \label{estimate c}
\end{equation}
\end{proposition}

\begin{proof}
Consider
\begin{align*}
\left\vert c(\omega ,x;h)-c_{N}(\omega ,x)\right\vert  &\leq \varepsilon
\int_{0}^{x}\left\vert \cos \omega t\,\right\vert dt=\frac \varepsilon 2
\int_{0}^{x}\left\vert e^{i\operatorname{Re}\omega\cdot t}e^{-\operatorname{Im}\omega\cdot t}+e^{-i\operatorname{Re}\omega\cdot t}e^{\operatorname{Im}\omega\cdot t} \right\vert dt \\
&\leq \frac \varepsilon 2 \int_{0}^{x}\left( e^{\operatorname{Im}\omega\cdot t}
+ e^{-\operatorname{Im}\omega\cdot t} \right) dt =
\varepsilon \int_{0}^{x}\cosh \left( |\operatorname{Im}\omega| \cdot t\right) dt =\frac{\varepsilon\sinh\left(|\operatorname{Im}\omega|\cdot x\right)}{|\operatorname{Im}\omega|}.
\end{align*}
Since the function $\sinh (\xi x)/\xi $ is monotonically increasing with
respect to both variables when $\xi ,x\geq 0$, we obtain the required
inequality (\ref{estimate c}).
\end{proof}

A similar statement is true for the solution $s(\omega ,x;\infty )$.

\begin{proposition}
Let the parameter $\omega $ belong to the strip $\left\vert \operatorname{Im}\omega
\right\vert \leq C$ where $C$ is a positive number and $\left\vert \omega
\right\vert >1$. Suppose that $\max_{\overline{S}}\left\vert
S_{f}(x,t)-S_N(x,t)\right\vert \leq \varepsilon $. Then
\begin{equation}
\left\vert s(\omega ,x;\infty )-s_{N}(\omega ,x)\right\vert \leq
\varepsilon \sinh (Cx)/C.  \label{estimate s}
\end{equation}%
If $\left\vert \omega \right\vert \leq 1$, then \[
\left\vert s(\omega,x;\infty)-s_{N}(\omega,x)\right\vert \leq\varepsilon
c_{b}x
\]
where the constant $c_{b}$ depends only on $b$. For any $\omega\neq0$ the
following estimate holds
\begin{equation}
\left\vert s(\omega,x;\infty)-s_{N}(\omega,x)\right\vert \leq\frac
{\varepsilon}{\left\vert \omega\right\vert }\sinh(Cx)/C.
\label{estimate s/omega}%
\end{equation}
\end{proposition}

\begin{proof}
Let $\left\vert \operatorname{Im}\omega \right\vert \leq C$ and $\left\vert \omega
\right\vert >1$. Then
\begin{equation*}
\left\vert s(\omega ,x;\infty )-s_{N}(\omega ,x)\right\vert  \leq
\left\vert \omega \right\vert \left\vert s(\omega ,x;\infty )-s_{N}(\omega
,x)\right\vert  \leq \varepsilon \int_{0}^{x}\left\vert \sin \omega t\,\right\vert dt.
\end{equation*}%
Now following the reasoning from the proof of the preceding proposition we
obtain (\ref{estimate s}).

Considering the case $\left\vert \omega \right\vert \leq 1$ we observe that
the function $\sin \left( \omega t\right) /\omega $ is analytic with respect
to $\omega $ and hence $\max_{\left\vert \omega \right\vert \leq
1}\left\vert \sin \left( \omega t\right) /\omega \right\vert
=\max_{\left\vert \omega \right\vert =1}\left\vert \sin \left( \omega
t\right) /\omega \right\vert =\max_{\left\vert \omega \right\vert
=1}\left\vert \sin \left( \omega t\right) \right\vert $. Denote this number
by $c(t)$. Again, due to the maximum principle we obtain $c(t)\leq
c(b)=:c_{b}$. Thus, $\left\vert s(\omega ,x;\infty )-s_{N}(\omega
,x)\right\vert \leq \varepsilon \int_{0}^{x}\left\vert \frac{\sin \omega t}{%
\omega }\,\right\vert dt\leq \varepsilon c_{b}x$. The estimate (\ref{estimate s/omega}) is proved in a complete analogy with (\ref{estimate c}).
\end{proof}

In order to be able to consider problems for equation (\ref{SLlambda}) with
boundary conditions involving the derivative of the solution we need to
obtain a convenient representation for it as well. Let $u_{N}$ be an
approximation of a solution $u$ of (\ref{SLlambda}), defined by the formula%
\begin{equation}
u_{N}(x)=\widetilde{u}(x)+\int_{-x}^{x}K_{f,N}(x,t)\widetilde{u}(t)dt
\label{uN}
\end{equation}%
where $\widetilde{u}$ is a linear combination of the functions $\sin (\omega
x)/\omega $ and $\cos (\omega x)$, and $K_{f,N}$ has the form (\ref{K(x,t)}%
). The corresponding exact solution has the form
\begin{equation*}
u(x)=\widetilde{u}(x)+\int_{-x}^{x}\mathbf{K}_{f}(x,t)\widetilde{u}(t)dt.
\end{equation*}%
From (\ref{uN}) it is quite easy to obtain a corresponding expression for $%
u_{N}^{\prime }$ observing that the differentiation of (\ref{K(x,t)}) where
the functions $u_{k}(x,t)$ are defined by (\ref{um}) does not present any
difficulty. Nevertheless one still has to prove a convenient estimate for
the difference $\left\vert u^{\prime }-u_{N}^{\prime }\right\vert $ which is
an additional task. To make it easier we choose here another way which
involves the transmutation operator for the Darboux associated Schr\"{o}%
dinger equation which in our notations is the operator $T_{1/f}$. In \cite%
{KrT2012} this operator was studied in detail and two explicit formulae for
the corresponding integral kernel $\mathbf{K}_{1/f}(x,t)$ in terms of $%
\mathbf{K}_{f}(x,t)$ were obtained (a third formula was presented in \cite%
{KT Transmut}).

Consider the function
\begin{equation}
v:=f\left( \frac{u}{f}\right) ^{\prime }=u^{\prime }-\frac{f^{\prime }}{f}u
\label{v}
\end{equation}%
which is a result of the Darboux transformation applied to $u$ (more on the
Darboux transformation see, e.g., \cite{Matveev} and \cite{Rosu}). The
function $v$ is a solution of the equation
\begin{equation*}
v^{\prime \prime }-q_{D}(x)v=\lambda v
\end{equation*}%
with $q_{D}=-q+2\left( f^{\prime }/f\right) ^{2}$. As was shown in \cite%
{KrT2012}, $v$ can be represented in the form
\begin{equation}
v=T_{1/f}\widetilde{v}  \label{vT1/f}
\end{equation}%
where $\widetilde{v}$ is a solution of the equation $\widetilde{v}^{\prime
\prime }=\lambda \widetilde{v}$, that is, $\widetilde{v}$ is a linear
combination of the functions $\sin (\omega x)/\omega $ and $\cos (\omega x)$
with $\omega ^{2}=-\lambda $,
\begin{equation*}
\widetilde{v}=\alpha \cos (\omega x)+\beta \sin (\omega x)/\omega .
\end{equation*}%
Assuming that
\begin{equation}
u(x)=ac(\omega ,x;h)+bs(\omega ,x;\infty )  \label{sol u}
\end{equation}
or, what is the same $\widetilde{u}=a\cos (\omega x)+b\sin (\omega x)/\omega
$, let us find the coefficients $\alpha $ and $\beta $ in terms of $a$ and $%
b $. From (\ref{v}) we have
\begin{equation}
v(x)=a\left( c^{\prime }(\omega ,x;h)-\frac{f^{\prime }}{f}c(\omega
,x;h)\right) +b\left( s^{\prime }(\omega ,x;\infty )-\frac{f^{\prime }}{f}%
s(\omega ,x;\infty )\right)  \label{v_ab}
\end{equation}%
meanwhile from (\ref{vT1/f}) we obtain
\begin{equation*}
v(x)=T_{1/f}\left[ \alpha \cos (\omega x)+\frac{\beta }{\omega }\sin (\omega
x)\right] .
\end{equation*}%
The last relation can obviously be written as follows
\begin{equation}
v(x)=T_{1/f}\left[ \left( \frac{\alpha }{\omega }\sin (\omega x)-\frac{\beta
}{\omega ^{2}}\cos (\omega x)\right)'\right] .  \label{vT1/fd}
\end{equation}%
In \cite{KrT2012} the following useful operator equality was obtained
\begin{equation*}
\frac{1}{f}T_{1/f}\frac{d}{dx}=\frac{d}{dx}\frac{1}{f}T_{f}
\end{equation*}%
which is true on $C^{1}[-b,b]$. Applying it to (\ref{vT1/fd}) we obtain%
\begin{eqnarray*}
v(x) &=&f(x)\frac{d}{dx}\left( \frac{1}{f(x)}T_{f}\left[ \frac{\alpha }{%
\omega }\sin (\omega x)-\frac{\beta }{\omega ^{2}}\cos (\omega x)\right]
\right) \\
&=&f(x)\frac{d}{dx}\left( \frac{1}{f(x)}\left[ \alpha s(\omega ,x;\infty )-%
\frac{\beta }{\omega ^{2}}c(\omega ,x;h)\right] \right) .
\end{eqnarray*}%
Comparison of this result with (\ref{v_ab}) gives us the relations $\alpha =b$ and $\beta =-\omega ^{2}a$.
Hence
\begin{equation}
v(x)=T_{1/f}\left[ b\cos (\omega x)-a\omega \sin (\omega x)\right] .
\label{vT1}
\end{equation}%
Notice that $v(0)=b$ and $v^{\prime }(0)=-a\omega ^{2}-bh$. This follows
from the properties of the operator $T_{1/f}$ (Remark \ref{RemTh}) and
observation that the value of $\left( 1/f\right) ^{\prime }$ in the origin
is $-h$.

From (\ref{vT1}) we obtain a convenient representation for the derivative of
the solution (\ref{sol u}),%
\begin{equation}
u^{\prime }=v+\frac{f^{\prime }}{f}u=T_{1/f}\left[ b\cos (\omega x)-a\omega
\sin (\omega x)\right] +\frac{f^{\prime }}{f}T_{f}\left[ a\cos (\omega x)+%
\frac{b}{\omega }\sin (\omega x)\right] .  \label{uprime}
\end{equation}

An approximation $K_{1/f,N}$ of the kernel $\mathbf{K}_{1/f}$ of
the operator $T_{1/f}$ can be done repeating the general scheme of Theorem \ref{Th Kapprox} for the Darboux associated potential $q_D$ and the particular solution $1/f$. However it is possible to omit the solution of another approximation problem, an approximation $K_{1/f,N}$ can be taken in the form
\begin{equation}
K_{1/f,N}=-\left( b_{0}v_{0}+\sum_{n=1}^{N}\left(
a_{n}v_{2n}+b_{n}v_{2n-1}\right) \right)  \label{K1/f}
\end{equation}%
with $b_{0}=-\mathbf{K}_{1/f}(0,0)=h/2$ and the coefficients $a_{n}$, $b_{n}$%
, $n=1,\ldots ,N$ from (\ref{K(x,t)}). Here $v_{k}$ are introduced as follows%
\begin{equation*}
v_{0}=\psi _{0}(x),\quad v_{2n-1}(x,t)=\sum_{\text{even }k=0}^{n}\binom{n}{k}%
\psi _{n-k}(x)t^{k},\quad v_{2n}(x,t)=\sum_{\text{odd }k=1}^{n}\binom{n}{k}%
\psi _{n-k}(x)t^{k}.  
\end{equation*}%
The fact that a representation for $K_{1/f,N}$ is a linear combination of the terms $v_{k}$
follows from Theorem \ref{Th Kapprox} where instead of $q$ and $f$ one
should consider $q_{D}$ and $1/f$ respectively. Then the corresponding
generalized wave polynomials $u_{k}$ result to be precisely $v_{k}$.
In (\ref{K1/f}) we state additionally that taking in
the representation the coefficients  from (\ref{K(x,t)}) we obtain an
approximation of the kernel $\mathbf{K}_{1/f}$. In order to
obtain this result one needs to consider the kernels $K_{f,N}$ and $
K_{1/f,N} $ as scalar components of a single bicomplex function and take
into account that the generalized wave polynomials (see \cite{KKTT}) $u_{k}$
and $v_{k}$ are nothing but scalar components of hyperbolic pseudoanalytic
formal powers (for the corresponding details we refer to \cite{KT Transmut}). Thus, the expression (\ref{K1/f}) is in fact a metaharmonic conjugate of (\ref{K(x,t)}).
We formulate here the result stating that if $\left\Vert \mathbf{K}_{f}-K_{f,N}\right\Vert
<\varepsilon $ then necessarily there is an appropriate estimate for $%
\left\Vert \mathbf{K}_{1/f}-K_{1/f,N}\right\Vert $ where $K_{1/f,N}$ is
defined by (\ref{K1/f}) and $\left\Vert \mathbf{\cdot }\right\Vert $ is the
maximum norm. We give its proof in the
Appendix A.

\begin{theorem}
\label{Th Estimate for K1/f}Let $\max_{\overline{S}}\left\vert \mathbf{K}
_{f}-K_{f,N}\right\vert <\varepsilon $ where $K_{f,N}$ has the form \eqref{K(x,t)}. Then $\max_{\overline{S}}\left\vert \mathbf{K}_{1/f}-K_{1/f,N}
\right\vert <\varepsilon C$ where $K_{1/f,N}$ is defined by \eqref{K1/f} and
the constant $C$ depends only on $f$ and $b$.
\end{theorem}

The estimates for the kernels $\mathbf{K}_{f}$ and $\mathbf{K}_{1/f}$ imply
corresponding estimates for the kernels $S_{f}$, $C_{f}$ and $S_{1/f}$,
$C_{1/f}$.

Theorem \ref{Th Estimate for K1/f} together with the equality (\ref{uprime})
suggests to approximate the derivatives of the solutions $c^{\prime}%
(\omega,x;h)$ and $s^{\prime}(\omega,x;\infty)$ by the functions
\begin{equation*}
\overset{\circ}{c}_{N}(\omega,x)  :=-\omega T_{1/f,N}\left[\sin\omega
x\right]+\frac{f^{\prime}}{f}T_{f,N}\left[\cos\omega x\right]  =-\omega^{2}s_{1/f,N}(\omega,x)+\frac{f^{\prime}}{f}c_{f,N}(\omega,x)
\end{equation*}
and
\begin{equation*}
\overset{\circ}{s}_{N}(\omega,x)  :=T_{1/f,N}\left[\cos\omega x\right]+\frac
{f^{\prime}}{f}T_{f,N}\left[\frac{\sin\omega x}{\omega}\right]  =c_{1/f,N}(\omega,x)+\frac{f^{\prime}}{f}s_{f,N}(\omega,x),
\end{equation*}
respectively. Notice that $c_{1/f,N}(\omega,x)$ is an approximation of
$c_{1/f}(\omega,x;-h)$.

Let us emphasize that $\overset{\circ}{c}_{N}(\omega,x)$ and $\overset{\circ
}{s}_{N}(\omega,x)$ do not coincide in general with the derivatives of
$\overset{}{c}_{N}(\omega,x)$ and $\overset{}{s}_{N}(\omega,x)$.

Approximation of the transmutation kernels corresponding to $f$ and $1/f$
imply approximations for the solutions $c(\omega,x;h)$, $s(\omega,x;\infty)$,
$c_{1/f}(\omega,x;-h)$ and $s_{1/f}(\omega,x;\infty)$. From the corresponding
estimates it is easy to obtain estimates for the approximations of $c^{\prime
}(\omega,x;h)$ and $s^{\prime}(\omega,x;\infty)$ by $\overset{\circ}{c}%
_{N}(\omega,x)$ and $\overset{\circ}{s}_{N}(\omega,x)$.

We note that due to (\ref{K1/f}), $\overset{\circ}{c}_{N}(\omega,x)$ and
$\overset{\circ}{s}_{N}(\omega,x)$ can be written in the following form
\begin{align*}
\overset{\circ}{c}_{N}(\omega,x)  & =-\omega\sin\omega x+2\omega\sum_{n=1}%
^{N}a_{n}\sum_{\text{odd }k=1}^{n}\binom{n}{k}\psi_{n-k}(x)\int_{0}^{x}%
t^{k}\sin\omega t\,dt\\
& \qquad +\frac{f^{\prime}}{f}\left(  \cos\omega x+2\sum_{n=0}^{N}a_{n}%
\sum_{\text{even }k=0}^{n}\binom{n}{k}\varphi_{n-k}(x)\int_{0}^{x}t^{k}%
\cos\omega t\,dt\right)
\end{align*}
and
\begin{align*}
\overset{\circ}{s}_{N}(\omega,x)  & =\cos\omega x-2\sum_{n=0}^{N}b_{n}%
\sum_{\text{even }k=0}^{n}\binom{n}{k}\psi_{n-k}(x)\int_{0}^{x}t^{k}\cos\omega
t\,dt\\
& \qquad +\frac{f^{\prime}}{\omega f}\left(  \sin\omega x+2\sum_{n=1}^{N}b_{n}%
\sum_{\text{odd }k=1}^{n}\binom{n}{k}\varphi_{n-k}(x)\int_{0}^{x}t^{k}%
\sin\omega t\,dt\right)
\end{align*}
where $b_{0}=h/2$.

\section{Numerical results}\label{Section7}
\subsection{General scheme and implementation details}\label{SubsectAlgorithm}
Consider a Sturm-Liouville equation
\begin{equation}\label{SLMain}
    -y''+q(x)y=\lambda y
\end{equation}
on a segment $[0,b]$ and a corresponding initial value problem
\begin{equation}\label{SLIC}
    y(0)=y_0\qquad\text{and}\qquad y'(0)=y_1
\end{equation}
or a spectral problem
\begin{align}
    \alpha_0 y(0)+\beta_0 y'(0) &= 0,\label{SLBC0}\\
    \alpha_b y(b)+\beta_b y'(b) &= 0,\label{SLBCb}
\end{align}
where we allow for the coefficients $\alpha_0$, $\beta_0$, $\alpha_b$ and $\beta_b$ to be not only constants but also entire functions of the square root $\omega$ of the spectral parameter $\lambda$ satisfying $|\alpha_0|+|\beta_0|\ne 0$ and $|\alpha_b|+|\beta_b|\ne 0$ (for every $\lambda$).

Based on the results of the previous sections we can formulate the following algorithm for solving initial value and spectral problems \eqref{SLIC} and \eqref{SLBC0}--\eqref{SLBCb} for equation \eqref{SLMain}.
\begin{enumerate}
\item Find a non-vanishing on $[0,b]$ solution $f$ of the equation
\[
-f''+q(x)f=0.
\]
Let $f$ be normalized as $f(0)=1$ and define $h:=f'(0)$.
\item Compute the functions $\varphi_k$ and $\psi_k$, $k=0,\ldots,N$ using \eqref{phik} and \eqref{psik}.
\item Compute the functions $\mathbf{c}_k$ and $\mathbf{s}_k$, $k=0,\ldots,N$ using \eqref{cm} and \eqref{sm}.
\item Find coefficients $a_0,a_1,\ldots,a_N$ and $b_1,\ldots,b_N$ of an approximation of the functions $\frac h2+\frac 14 \int_0^x q(s)\,ds$ and $\frac 14 \int_0^x q(s)\,ds$ by linear combinations $\sum_{n=0}^N a_n\mathbf{c}_n(x)$ and $\sum_{n=1}^N b_n\mathbf{s}_n(x)$ as in Theorem \ref{Th Kapprox}. Set $b_0=a_0$. Note that also one can take $a_0=\frac h2$ and approximate the function $\frac h2(1-f(x))+\frac 14 \int_0^x q(s)\,ds$ by a linear combination $\sum_{n=1}^N a_n\mathbf{c}_n(x)$ in order to find coefficients $a_1,\ldots,a_N$.
\item Calculate the approximations $s_N(\omega,x)$ and $c_N(\omega,x)$ of the solutions $s(\omega,x;\infty)$ and $c(\omega,x;h)$ by \eqref{sN} and \eqref{cN}. If necessary, calculate the approximations of the derivatives of the solutions using \eqref{uprime} and \eqref{K1/f}. Recall that the expressions $T_{1/f} \cos \omega t$ and $T_{1/f} \frac{\sin \omega t}\omega$ can be computed similarly to \eqref{cN} and \eqref{sN} using the coefficients $\tilde a_n:=-b_n$ and $\tilde b_n:=-a_n$ and the functions $\psi_n$ instead of the functions $\varphi_n$, c.f., \eqref{K(x,t)} and \eqref{K1/f}.
\item According to \eqref{ICcos} and \eqref{ICsin} the approximation of the solution of the initial problem \eqref{SLIC} has the form
    \[
    y=y_0 c_N(\omega,x)+(y_1-y_0h)s_N(\omega, x).
    \]
    The eigenvalues of the problem \eqref{SLBC0}--\eqref{SLBCb} coincide with the squares of the zeros of the entire function
    \begin{equation}\label{SLCharEq}
        \Phi(\omega) := \alpha_b\bigl(\beta_0 c(\omega, b;h)-(\alpha_0+\beta_0 h)s(\omega,b;\infty)\bigr)+\beta_b\bigl(\beta_0 c'(\omega,b;h)-(\alpha_0+\beta_0 h)s'(\omega,b;\infty)\bigr)
    \end{equation}
    and are approximated by squares of zeros of the function
    \begin{equation}\label{SLCharEqApprox}
        \Phi_N(\omega) := \alpha_b\bigl(\beta_0 c_N(\omega, b)-(\alpha_0+\beta_0 h)s_N(\omega,b)\bigr)+\beta_b\bigl(\beta_0 c_N'(\omega,b)-(\alpha_0+\beta_0 h)s_N'(\omega,b)\bigr).
    \end{equation}
    Note that despite the division by $\omega$ in \eqref{Int_tsin} and \eqref{Int_tcos} the singularity at zero of the function $\Phi_N(\omega)$ is removable and $\Phi_N(\omega)$ can be considered as an entire function.
\item The eigenfunction $y_\lambda$ corresponding to the eigenvalue $\lambda=\omega^2$ can be taken in the form
    \begin{equation}\label{SLEigenfunction}
        y_\lambda=\beta_0 c(\omega,x;h)-(\alpha_0+\beta_0 h)s(\omega,x;\infty).
    \end{equation}
    Hence once the eigenvalues are calculated the computation of the corresponding eigenfunctions can be done using formulas \eqref{sN} and \eqref{cN}.
\end{enumerate}

The results of the previous section allow us to prove the uniform error bound for all approximate zeros of the characteristic function (at least when the coefficients in the boundary conditions \eqref{SLBC0} and \eqref{SLBCb} are independent of the spectral parameter) obtained by the proposed algorithm and that neither spurious zeros appear nor zeros missed whenever inequalities \eqref{KxxErr} and \eqref{KxmxErr} are satisfied with sufficiently small $\varepsilon_1$ and $\varepsilon_2$. For not going into too much detail in the present paper we consider only the case of Dirichlet boundary conditions, i.e., when the characteristic equation reduces to $s(\omega, b;\infty)=0$. We also refer the reader to \cite{HrynivMykytyuk2009} where similar questions are discussed.

\begin{proposition}\label{Prop Uniform Errors}
Suppose that the boundary conditions \eqref{SLBC0} and \eqref{SLBCb} are the Dirichlet boundary conditions, i.e., $\alpha_0=\alpha_b\equiv 1$, $\beta_0=\beta_b\equiv 0$. Then for every $\varepsilon>0$ there exist such $\varepsilon_{1,2}>0$ that if inequalities \eqref{KxxErr} and \eqref{KxmxErr}  are satisfied for some $N$ with these $\varepsilon_{1}$ and $\varepsilon_{2}$
respectively then all the zeros (including multiplicities) of the characteristic function of the problem \eqref{SLMain}, \eqref{SLBC0}, \eqref{SLBCb} are approximated by the (complex) zeros of the function $\Phi_N(\omega)$ with errors uniformly bounded by $\varepsilon$ and no spurious zeros appear.
\end{proposition}

\begin{remark}
Even in the case when the problem \eqref{SLMain}, \eqref{SLBC0}, \eqref{SLBCb} possesses a purely real spectrum we need to consider complex zeros of the function $\Phi_N(\omega)$ for Proposition \ref{Prop Uniform Errors} to hold. For zeros of multiplicity greater than one the proposition establishes that in an $\varepsilon$-neighbourhood of such zero there is a corresponding number of zeros of the approximate characteristic function $\Phi_N$.
\end{remark}

\begin{proof}
For the Dirichlet boundary conditions the characteristic function has the form $\Phi(\omega)=s(\omega,b;\infty)$. Consider the function $\widetilde \Phi(\omega):=\omega \Phi(\omega)$. It is known that $\widetilde \Phi(\omega)$ is an entire function (see, e.g., \cite[\S 1.3]{Marchenko}), has a countable set of zeros, all of finite multiplicity. Denote this set of zeros by $\Omega$.

Let $0<\varepsilon<\frac 1{2b}$ is given. Define a number
\begin{equation*}
    \widetilde\varepsilon=\min\left\{\varepsilon,\inf_{\omega_1,\omega_2\in\Omega,\ \omega_1\ne \omega_2}\frac{|\omega_1-\omega_2|}2\right\}.
\end{equation*}
Since $\Omega$ has no finite accumulation point, $\widetilde\varepsilon>0$. Note that disks of radiuses $\widetilde\varepsilon$ and centers in different zeros of the function $\widetilde\Phi$ do not have common interior points.

Now we show that
\begin{equation}\label{Eq Inf Char Eq Val}
    m:=\inf\left\{\widetilde\Phi(z): z\in\mathbb{C},\ |z-\omega|=\widetilde\varepsilon,\ \omega\in\Omega\right\}>0.
\end{equation}
Recall that
\begin{equation}\label{Eq Char Fun}
    \widetilde\Phi(\omega) = \sin\omega b+\int_0^b S_f(b,t)\sin\omega t\,dt
\end{equation}
and that the zeros (excluding 0) of $\widetilde \Phi(\omega)$ after reordering satisfy the following asymptotics \cite[Lemma 1.3.3]{Marchenko}
\begin{equation}\label{Eq Eigenvalue Asympt}
    \omega_n=\frac{\pi}b n+\frac{\alpha_n}n,\qquad \text{where } \sup |\alpha_n|<\infty.
\end{equation}
Hence for large values of $n$ the circles $\{\omega:|\omega-\omega_n|=\widetilde\varepsilon,\ \omega_n\in\Omega\}$ belong to the rings $R_n:=\{\omega: \frac{\widetilde\varepsilon}2\le|\omega-\frac{\pi}b n|\le\frac{3\widetilde\varepsilon}2\}$. Since the function $|\sin\omega b|$ is periodic with the period $\frac{\pi}b$ and does not vanish on $\{\omega: \frac{\widetilde\varepsilon}2\le |\omega|\le \frac{3\widetilde\varepsilon}2\}$, there exists
\begin{equation}\label{Eq Min Sin}
    m_1:=\min\left\{ |\sin\omega b|: \frac{\widetilde\varepsilon}2<|\omega|<\frac{3\widetilde\varepsilon}2\right\} >0.
\end{equation}
Due to \cite[Lemma 1.3.1]{Marchenko} if $|\omega|\to\infty$ with $|\operatorname{Im}\omega|$ remaining bounded then $\int_0^b S_f(b,t)\sin\omega t\,dt\to 0$. Hence we obtain from \eqref{Eq Char Fun}, \eqref{Eq Eigenvalue Asympt} and \eqref{Eq Min Sin} that, e.g., $|\widetilde \Phi(\omega)|\ge \frac{m_1}{2}$ when $|\omega-\omega_n|=\widetilde\varepsilon$
for all sufficiently large $|n|$. For all remaining values of $n$ the function $\widetilde\Phi(\omega)$ does not vanish on the circles $|\omega-\omega_n|=\widetilde\varepsilon$, which finishes the proof of the positivity of the constant $m$ in \eqref{Eq Inf Char Eq Val}.

Due to the asymptotics \eqref{Eq Eigenvalue Asympt} all zeros $\omega_n$ belong to a strip $|\operatorname{Im}\omega|\le M$. Let $\varepsilon_1$ be such that
\begin{equation*}
    \varepsilon_1\frac{\sinh ( (M+\widetilde\varepsilon)b)}{M+\widetilde\varepsilon}\le m
\end{equation*}
and for some $N$ the inequalities \eqref{KxxErr} and \eqref{KxmxErr} are satisfied with this $\varepsilon_1$. Consider the approximate solution $s_N$.
Then it follows from \eqref{estimate s/omega} that on all circles $|\omega-\omega_n|=\widetilde\varepsilon$, $\omega_n\in\Omega$ we have
\begin{equation*}
    |\widetilde \Phi(\omega)-\omega \Phi_N(\omega)|=|\omega s(\omega,b;\infty)-\omega s_N(\omega,b)|<m.
\end{equation*}
Hence by the Rouche theorem the functions $\widetilde \Phi(\omega)$ and $\omega \Phi_N(\omega)$ have the same number of zeros in the disks $|\omega-\omega_n|<\widetilde\varepsilon$, $\omega_n\in\Omega$.

The statement that no spurious zeros appear follows from the results of \cite[\S 1.3]{Marchenko} where it is shown that the functions $s(\omega,b;\infty)$ and $\sin \omega b$ possess the same number of zeros in $\{\omega: |\operatorname{Re} \omega|<2n+1/2\}$ for sufficiently large $n$, and it can be seen that this statement holds for the function $s_N$ as well.
\end{proof}

\begin{remark}
An analogues statement can be proved for all other boundary conditions of the form \eqref{SLBC0}, \eqref{SLBCb}, at least whenever they are spectral parameter independent. The scheme of the proof remains the same and should involve corresponding asymptotic relations similar to \eqref{Eq Eigenvalue Asympt} which can also be found in \cite[\S 1.3]{Marchenko}.
\end{remark}

Some remarks should be made regarding the implementation of the described algorithm.

The non-vanishing solution of equation \eqref{SLhom} can be constructed using the SPPS representation, see, e.g., \cite{KrPorter2010} for details. In the case when an exact particular solution is known we compared the obtained approximated solution against the exact one.

The accuracy and speed of the calculation of the recursive integrals play a crucial role for the accuracy and speed of the proposed algorithm. Previously we applied two different approaches for the integration. One is based on a modification of the 6 point Newton-Cottes formula \cite{CKT2013}, second uses spline approximation and integration of the obtained splines \cite{KKB}. For the first method we can easily use several millions subdivision points, meanwhile the computation time required to construct the approximating splines limits the maximal number of subdivision points for the second methods to tens of thousands. Computation based on the first approach can be highlighted as an especially recommendable option. In all numerical tests reported recently (see, e.g., \cite{CKT2013}) it delivered fast and accurate results. However for the present work for all but one example we opted for another approach. The main reason is that both methods possess the saturation property, i.e., their accuracy depends polynomially on the used step size and are not suited well enough for really high-precision calculation. For example, for the 6 point Newton-Cottes formula the final accuracy is of the order $O(h^7)$, where $h$ is the step size.

Further choice between available methods is limited by the requirement that the integrals should be calculated recursively, which leads to the following simple condition. Either the integration method should take the values of the function $g$ defined in some predefined abscissas $x_0<x_1<\ldots<x_M$ and return the values of the indefinite integral in the same set of abscissas, or the integration method should determine the set of abscissas $x_0<x_1<\ldots<x_{M'}$ analyzing the given function, and after that provide the value of the indefinite integral in an arbitrary point of interest using only the values $g(x_0),\ldots,g(x_{M'})$.

From several known methods of evaluation of indefinite integrals with high accuracy and suitable for computing the recursive integrals, e.g., Clenshaw--Curtis, Sinc and double exponential methods \cite[Section 2.13.1]{DavisRabinovich}, \cite{Haber}, \cite{MuhammadMori}, \cite{Stenger}, \cite{TakahasiMori}, \cite{TSM}, \cite{Wright}, we chose the Clenshaw--Curtis quadrature scheme based on the approximation of the integrand by a partial sum of its expansion into a series in terms of Tchebyshev polynomials and termwise integration of the approximation. The method is described in detail in \cite[Section 2.13.1]{DavisRabinovich} and has the advantage that restricting all calculation to the Tchebyshev nodes $\frac b2\left(1+\cos\frac{k\pi}M\right)$, $k=0,\ldots, M$ it reduces to the discrete cosine transform (DCT), a simple transformation of the obtained coefficients and the inverse DCT, and hence has a near-linear complexity with respect to the number $M+1$ of used abscissas. Another advantage is that the method works for an arbitrary differentiable function, analyticity is not required. It should be mentioned that the smallest computation time is achieved when $M=2^m$ and that in some cases to obtain a good accuracy of the calculated recursive integrals we used extra digits for intermediate calculations, see the following examples for the details. Another possibility is to split the interval of integration into several subintervals.

To find the coefficients $a_1,\ldots,a_n$ and $b_1,\ldots,b_n$ of the approximations from Theorem \ref{Th Kapprox} we applied the least squares method. There exist other methods providing more accurate uniform approximations, however as a rule they are slower, and in our implementation of the described algorithm in Mathematica software even the build-in function \texttt{LeastSquares} required a computation time comparable to the time of the calculation of all recursive integrals.

For the calculation of the integrals \eqref{Int_tsin} and \eqref{Int_tcos} we used the recurrent formulas \cite[4.3.119 and 4.3.123]{AbramovitzStegun}.

We do not discuss in detail the step of finding the eigenvalues. The main purpose of the numerical examples is to illustrate that the final approximation of a characteristic equation contains all the information required to evaluate accurate approximations of the eigenvalues. The problem reduces to the search of zeros of some analytic function with the only possible pole at $\omega=0$. The derivative of this function is easily obtainable and in the most complicated cases, say clusters of closely located eigenvalues as, e.g., in the Coffey-Evans problem (Example \ref{ExCE}), well-known theorems of complex analysis like the argument principle are useful, see, e.g., \cite{YingKatz}, \cite{DelnitzEtAl}. It is possible that the calculation of the closest to zero eigenvalues by the proposed method may present difficulties due to the pole of $\Phi_N$ at $\omega=0$. One possible solution is to perform a spectral shift, that is, to consider equation \eqref{SLMain} in the form $-y''+(q(x)+\lambda_\ast)y=(\lambda+\lambda_\ast)y$, where $|\lambda_\ast|$ is sufficiently separated from zero. Another solution is to use the SPPS representation \cite{KrPorter2010}. The approximation of the characteristic equation given by the SPPS representation works especially well near the origin, and all required functions are calculated on the step 3 of the described algorithm.

By the described algorithm the functions $\varphi_n$ are calculated only in $M+1$ points coinciding with the Tchebyshev nodes allowing us to compute the eigenfunctions $u_\lambda$ directly by formulas \eqref{SLEigenfunction}, \eqref{sN} and \eqref{cN} only in these $M+1$ points. If for some applications such subset of points is insufficient, the functions $\varphi_n$ can be easily interpolated to arbitrary subset of the segment $[0,b]$. One of the best ways to perform the interpolation is by using the partial sums of approximations of the functions $\varphi_n$ by their expansions into series in terms of Tchebyshev polynomials. The expansion coefficients can be obtained using the DCT, and final interpolations are obtained by summing up corresponding partial sums. Since all the functions $\varphi_n$ are computed applying similar approximation procedure, described interpolation does not deteriorate significantly the accuracy. We illustrate such approach in Example \ref{ExPaine1} where we show that even large index highly oscillating eigenfunctions can be accurately approximated.

\subsection{Sturm-Liouville spectral problems}
\begin{example}\label{ExPaine1}
Consider the
following spectral problem (the first Paine problem, \cite{Paine})
\begin{equation*}
\begin{cases}
-u''+e^x u=\lambda u, & 0\le x\le \pi,\\
u(0,\lambda)=0, & u(\pi,\lambda)=0.
\end{cases}
\end{equation*}
With the help of Mathematica software we found a non-vanishing particular solution
\begin{equation}\label{Ex1PS}
u_0(x)=I_0\big(2e^{x/2}\big)
\end{equation}
and the characteristic function
\[
\Phi(\omega)=I_{2 i \omega }(2)I_{-2 i \omega }\big(2
   \sqrt{e^\pi}\big)-I_{-2 i \omega }(2) I_{2 i
   \omega }\big(2 \sqrt{e^\pi}\big),
\]
where $\omega^2=\lambda$ and $I$ is the modified Bessel function of the first kind.

In this example we performed calculations in Matlab in machine precision and in Mathematica using high precision arithmetic. For the Matlab program we used $N=30$ for the approximations \eqref{KxxErr} and \eqref{KxmxErr} and computed all involved recursive integrals using Newton-Cottes integration formula with $M=20000$. This experiment is similar to \cite[Example 9]{KT Transmut}, however the runtime improved due to the different integration method used. The approximation errors achieved in \eqref{KxxErr} and \eqref{KxmxErr} were $5.5\cdot 10^{-11}$ and $9.3\cdot 10^{-11}$ respectively. 500 eigenvalues were calculated and the maximal absolute error of the approximated eigenvalues was $1.95\cdot 10^{-9}$. The overall time required for the calculation was 5 seconds for the approximation part (Steps 1--5 of the algorithm from Subsection \ref{SubsectAlgorithm}) and 6.5 seconds was required to finish Step 6, such time was necessary because we constructed a spline approximating the characteristic function and used Matlab  function \texttt{fnzeros} to find its zeros. A personal computer equipped with Intel i7-3770 processor was used for this and following computations.

For the second experiment we performed all the numerical calculations with 200-digit arithmetic in Mathematica. We used $M=256$ for the calculation of all involved recursive integrals. The particular solution was computed using the SPPS representation with $150$ formal powers and compared with the solution  \eqref{Ex1PS} to verify the precision of the approximate solution. The maximal difference between the approximate and the exact solutions was $3.7\cdot 10^{-187}$ showing an excellent accuracy achievable by the combination of the SPPS representation and the Clenshaw--Curtis integration procedure.

\begin{figure}[htb]
\centering
\includegraphics[
height=2.19in,
width=5.0in
]
{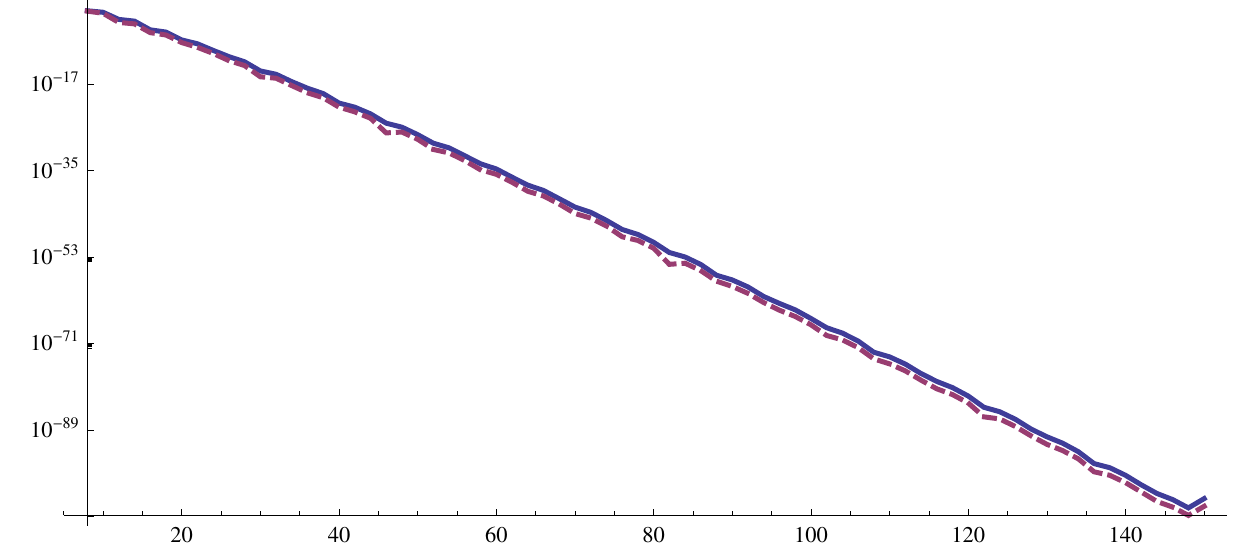}
\caption{The graphs of the maximal approximation error in \eqref{KxxErr}
and \eqref{KxmxErr} (solid line) and the maximal relative error of the first $500$ eigenvalues (dashed line) of the first Paine problem (Example \ref{ExPaine1}) as functions of $N$.}
\label{ExPaineApproxError}
\end{figure}

Using the obtained particular solution we calculated the functions $\mathbf{c}_n$ and $\mathbf{s}_n$ for $n\le 150$. After that for each $N=8,10,\ldots,150$ we found coefficients $a_0,\ldots,a_N$, $b_1,\ldots,b_N$ for \eqref{KxxErr} and \eqref{KxmxErr}, calculated first 500 eigenvalues as zeros of the approximate characteristic function and compared them to the exact ones. The function \texttt{FindRoot} from Mathematica was used to find both exact and approximate eigenvalues. On Figure \ref{ExPaineApproxError} we present both the obtained approximation errors in the inequalities \eqref{KxxErr} and \eqref{KxmxErr} and the maximal resulted relative error of the first 500 eigenvalues. We would like to point out that the error decays exponentially with respect to the number $N$ of functions used and that the resulted error of the eigenvalues is bounded by the error of the approximations on the characteristics.

\begin{figure}[htb]
\centering
\includegraphics[
height=2.05in,
width=5.0in
]
{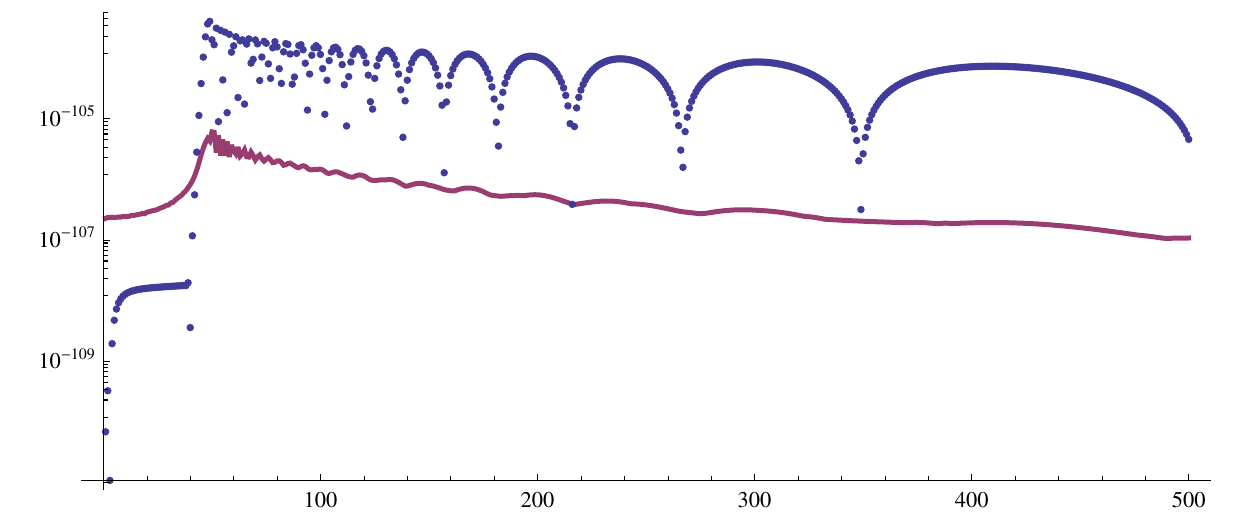}
\caption{The graph of the absolute error of the first $500$
eigenvalues (thick dots) and 500 corresponding eigenfunctions (line) of the first Paine problem (Example \ref{ExPaine1}) obtained from the described algorithm with $N=148$.}
\label{ExPaineEigError}
\end{figure}

Note that on the final step from $N=148$ to $N=150$ the approximation error increases, which can be explained by the fact that we passed a limit where the used precision and the number of points work well. Hence we used the value $N=148$ to verify the accuracy of the first 500 eigenvalues and 500 corresponding eigenfunctions. We calculated the eigenfunctions $u_\lambda$ satisfying the initial condition $u'_\lambda(0)=\sqrt{\lambda}$ on the uniform mesh of 2000 points from $[0,\pi]$ using the interpolated values of the functions $\varphi_n$ and compared them to the exact ones.
The approximation errors in \eqref{KxxErr} and \eqref{KxmxErr} were $3.9\cdot 10^{-106}$ and $6.1\cdot 10^{-106}$ respectively, while the largest error of the computed eigenvalues was $4.0\cdot 10^{-104}$ for the eigenvalue number 49. The computation time was 56 seconds for the approximation part (Steps 1--5 of the algorithm from Subsection \ref{SubsectAlgorithm}) and 213 seconds was required for build-in Mathematica function \texttt{FindRoot} to finish Step 6. On Figure \ref{ExPaineEigError} we show the absolute errors of the first 500 eigenvalues together with the distances between exact and approximate eigenfunctions in the uniform norm. Observe that even though the eigenvalues grow, the absolute errors remain essentially of the same order. The computed eigendata accuracy does not deteriorate even for eigenvalues with larger indices. For example, absolute errors of $\lambda_{1000}$, $\lambda_{2500}$ and $\lambda_{10000}$ are $1.3\cdot 10^{-105}$, $2.1\cdot 10^{-105}$ and $2.9\cdot 10^{-106}$, and errors of corresponding eigenfunctions (evaluated on a mesh of 25000 points) are $4.7\cdot 10^{-108}$, $1.3\cdot 10^{-108}$ and $1.8\cdot 10^{-109}$, respectively.
\end{example}

\begin{example}\label{ExPaine2}
Consider the
following spectral problem (the second Paine problem, \cite{Paine, Pryce})
\begin{equation*}
\begin{cases}
-u''+\frac{1}{(x+0.1)^2} u=\lambda u, & 0\le x\le \pi,\\
u(0,\lambda)=0, & u(\pi,\lambda)=0.
\end{cases}
\end{equation*}
With the help of Mathematica software we found a non-vanishing particular solution
\begin{equation}\label{Ex2PS}
u_0(x)=(1+10x)^{(1+\sqrt{5})/2}
\end{equation}
and the characteristic function
\[
\Phi(\omega)=M_{0,-\frac{\sqrt{5}}{2}}\left(\frac{\omega}{5}\right)
   W_{0,-\frac{\sqrt{5}}{2}}\left(\frac{\omega}{5}(1+10 \pi )\right)-M_{0,-\frac{\sqrt{5}}{2}}\left(\frac{\omega}{5}
   (1+10 \pi )\right)
   W_{0,-\frac{\sqrt{5}}{2}}\left(\frac{\omega}{5}\right),
\]
where $\omega^2=\lambda$, $M$ and $W$ are the Whittaker functions \cite{AbramovitzStegun}.

We calculated an approximate particular solution in Mathematica using the SPPS representation with 300 formal powers and performed integrations with $M=1024$ and 200-digits arithmetic. The larger number of points compared to Example \ref{ExPaine1} was necessary for an accurate evaluation of the recursive integrals and possibly can be explained by the fact that the accuracy of the Clenshaw--Curtis integration depends on the decay rate of coefficients of the function expansion into a series in terms of Tchebyshev polynomials, which in turn is related to the size of an ellipse on a complex plane with foci at the points $z_1=-1$ and $z_2=1$ to which the potential $q$ possesses the analytic continuation, see, e.g., \cite{RiessJohnson}, \cite{Trefethen}. The maximum error of the approximate solution compared with the exact one \eqref{Ex2PS} was less than $4\cdot 10^{-164}$.

\begin{figure}[htb]
\centering
\includegraphics[
height=2.22in,
width=5.0in
]
{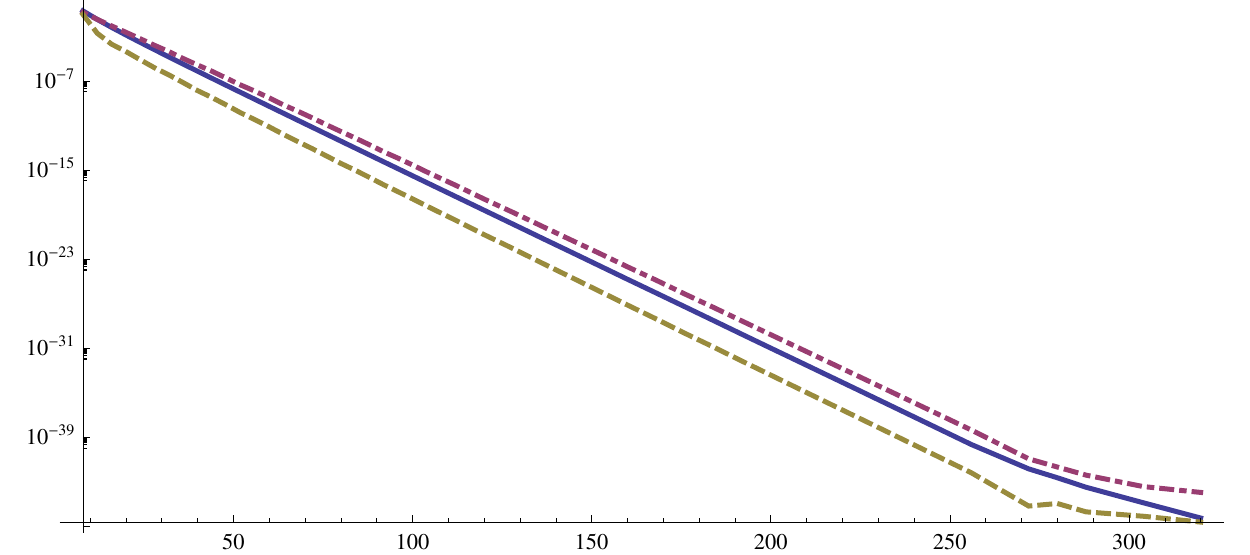}
\caption{The graphs of the maximal approximation error in \eqref{KxxErr}
and \eqref{KxmxErr} (solid line), the maximal absolute error (dot-dashed line) and the maximal relative error (dashed line) of the first $500$
eigenvalues of the second Paine problem (Example \ref{ExPaine2}) as functions of $N$.}
\label{ExPaine2ApproxError}
\end{figure}

Using the obtained particular solution we calculated the functions $\mathbf{c}_n$ and $\mathbf{s}_n$ for $n\le 320$. For the calculation we used 400-digit arithmetic. Such precision appeared to be a necessity for the build-in Mathematica function \texttt{LeastSquares} to be able to produce approximation coefficients for \eqref{KxxErr} and \eqref{KxmxErr} for large values of $N$. After that for values of $N$ in the range $[8,320]$ taken with step sizes increasing from $4$ to $16$ we found coefficients $a_0,\ldots,a_N$, $b_1,\ldots,b_N$ for \eqref{KxxErr} and \eqref{KxmxErr}. It turned out that finding  the first several eigenvalues as zeros of the approximated characteristic function $\Phi_N$ for large $N$ is not possible directly with Mathematica's function \texttt{FindRoot} (see the explanation at the end of subsection \ref{SubsectAlgorithm}), so we applied the following procedure. First we calculate roots of the polynomial obtained as a truncation of the SPPS representation of the characteristic equation, see, e.g., \cite{KrPorter2010}. These roots are known to give an excellent accuracy especially for the eigenvalues close to the origin. Then we find zeros of $\Phi_N$. To combine the two obtained sets we find the two closest values in these two sets and take the smaller ones from the roots of the SPPS polynomial and the larger ones from the zeros of $\Phi_N$. Such strategy worked well for all values of $N$. On Figure \ref{ExPaine2ApproxError} we present both the obtained approximation errors in the inequalities \eqref{KxxErr} and \eqref{KxmxErr} and the maximal resulted absolute and relative errors of the first 500 eigenvalues as functions of $N$. We would like to point out that the error decay exponentially with respect to the number $N$ of functions used and that the slopes of the graphs are close.

\begin{figure}[htb]
\centering
\includegraphics[
height=2.01in,
width=5.0in
]
{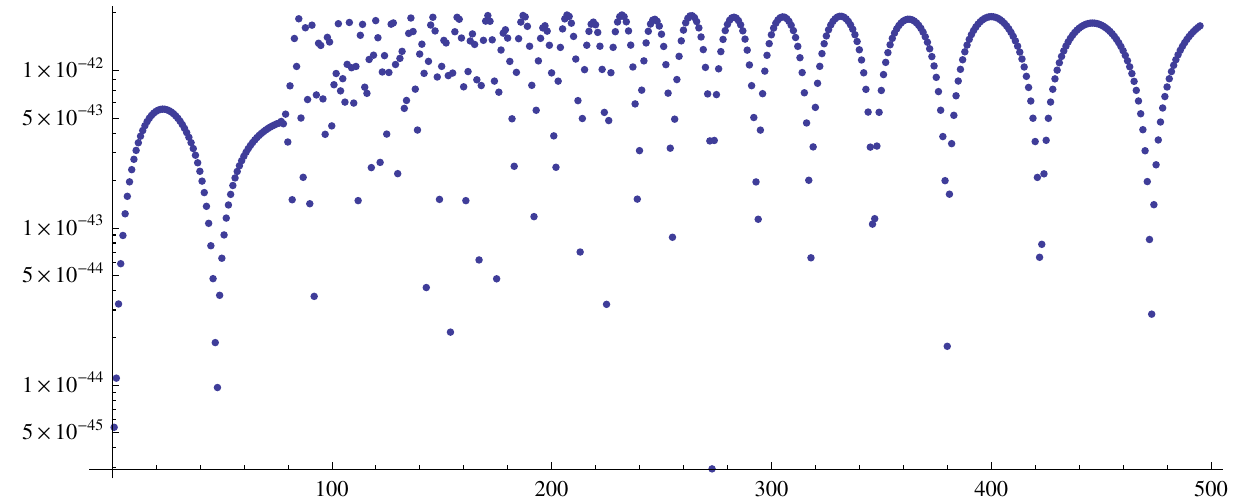}
\caption{The graph of the absolute error of the first $500$
eigenvalues of the second Paine problem (Example \ref{ExPaine2}) obtained from the described algorithm with $N=280$.}
\label{ExPaine2EigError}
\end{figure}

Note that the slope of the approximation error graph changes around $N=270$, and the absolute and relative errors decrease slower starting from this value of $N$. Again we can explain such behavior by the fact that we are close to the limit where the used precision and the number of points still work. Hence we used the value $N=280$ to present the graph of the errors of the first 500 eigenvalues. The approximation errors in \eqref{KxxErr} and \eqref{KxmxErr} were $2.5\cdot 10^{-43}$ and $2.3\cdot 10^{-43}$ respectively, while the largest error of the computed eigenvalues was $2.3\cdot 10^{-42}$ for the eigenvalue number 237. On Figure \ref{ExPaine2EigError} we show the absolute errors of the first 500 eigenvalues. Again we see that the absolute errors of the eigenvalues remain at the same level. Moreover, the accuracy of the computed eigenvalues does not deteriorate for even larger eigenvalues. For example, absolute errors of $\lambda_{1000}$, $\lambda_{2500}$ and $\lambda_{10000}$ are $1.4\cdot 10^{-42}$, $1.1\cdot 10^{-42}$ and $2.2\cdot 10^{-43}$ respectively.

While high precision arithmetic is necessary in this example to produce accurate eigenvalues, it is possible to use smaller parameters $N$ and $M$ and lower precision arithmetic if one looks for the eigenvalues accurate to 13--15 digits (i.e., with the precision expected from double-precision machine arithmetics), leading to faster runtime, comparable with other codes available. For example, we used $M=256$, $N=120$ and 128-digit arithmetic and the algorithm finished Steps 1--5 in 33 seconds and in 103 seconds found 500 eigenvalues with the largest absolute error of $2\cdot 10^{-13}$.
\end{example}

\begin{example}\label{ExCE}
Consider the Coffey-Evans problem \cite{ChildCmabers}
\begin{equation}
\label{CEeqn}%
\begin{cases}
-u''+\bigl(\beta^2\sin^2 2x-2\beta \cos 2x\bigr) u=\lambda u, & -\frac\pi 2\le x\le \frac\pi 2,\\
u\bigl(-\frac\pi 2,\lambda\bigr)=u\bigl(\frac\pi 2,\lambda\bigr)=0.
\end{cases}
\end{equation}
This problem is considered as a standard test case for numerical methods for solving Sturm-Liouville spectral problems, see, e.g., \cite{Pryce}, \cite{PruceFulton}, \cite{AliciTaceli}, \cite{KrPorter2010}, \cite{Ledoux2010}, and presents the challenge of distinguishing eigenvalues within the triple clusters which form as the parameter $\beta$ increases. The equation in \eqref{CEeqn} is a particular case of the Whittaker-Hill equation and its particular solution with the initial conditions $u(-\frac\pi 2)=1$, $u'(-\frac\pi 2)=0$ is known \cite{HemeryVeselov} and is given by
\begin{equation}\label{ExCEps}
u_0(x)=e^{\beta\cos 2x}.
\end{equation}
To our best knowledge the most accurate eigenvalues of \eqref{CEeqn} are reported in \cite{AliciTaceli}, where the table of the first 18 eigenvalues for the case $\beta=50$ correct to 24 decimal places is included. It is worth mentioning that the method used in \cite{AliciTaceli} is suitable only for a special subclass of Sturm-Liouville equations. For the numerical example we also chose $\beta=50$.

\begin{figure}[htb]
\centering
\includegraphics[
height=3.74in,
width=5.0in
]
{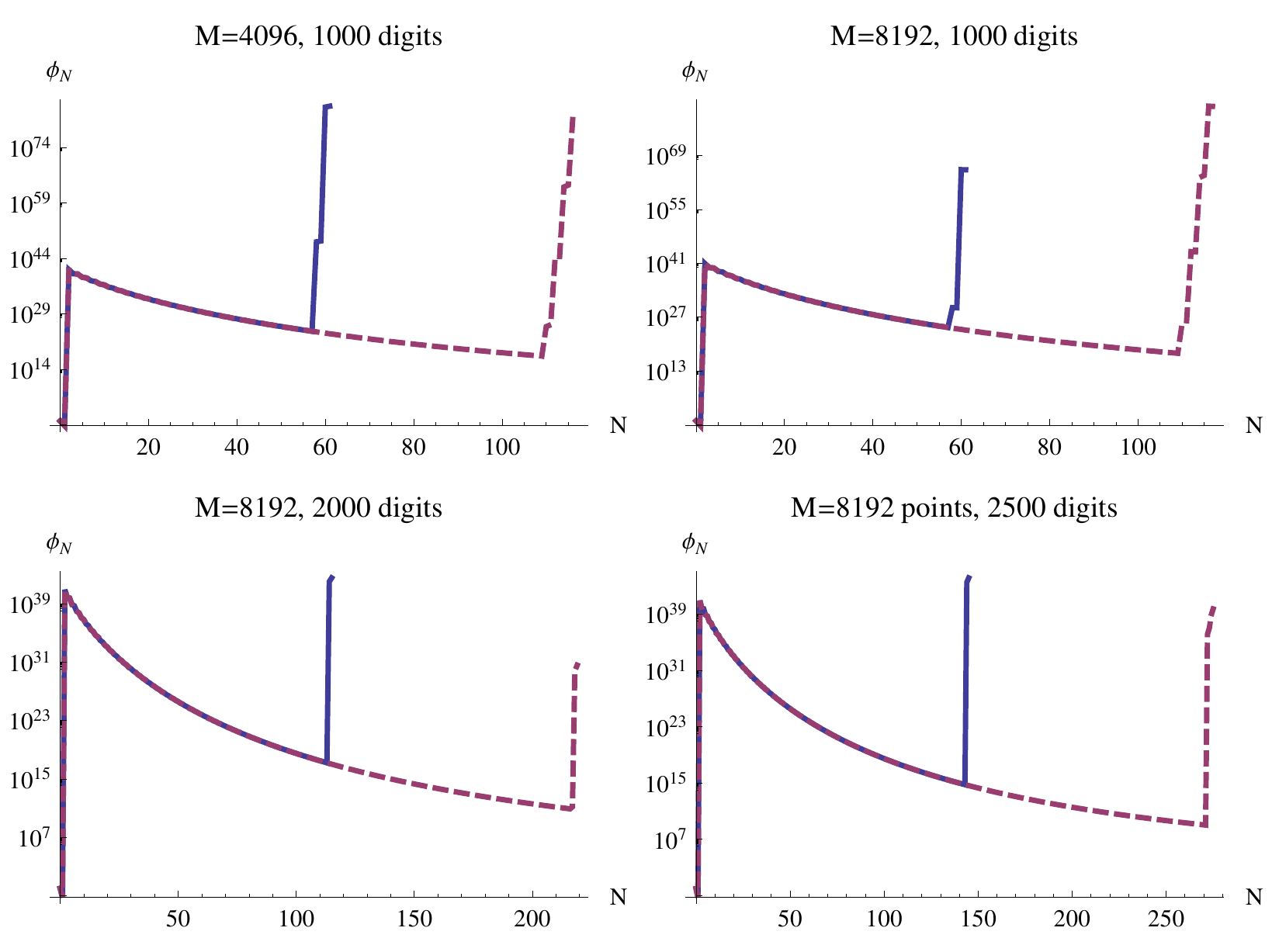}
\caption{The graphs of the absolute values of $\varphi_N(1)$ calculated with different number of points $M$ and different arithmetic precision used for computation of the recursive integrals for the Coffey-Evans problem (Example \ref{ExCE}). Solid line: functions $\varphi_N$ are calculated using \eqref{Xn}--\eqref{phik}, dashed line: functions $\varphi_N$ are calculated using formulas from \cite{KrTNewSPPS}.}
\label{ExCEphiError}
\end{figure}

An approximation of the particular solution \eqref{ExCEps} was calculated using the SPPS representation with 1200 formal powers. We used 800-digit arithmetic and $M=2048$ for the calculation of the formal powers. The resulted error of the approximate solution compared with \eqref{ExCEps} was less than $2.7\cdot 10^{-643}$. We would like to point out such remarkable accuracy. The built-in Mathematica's function \texttt{NDSolve} was not able to achieve any acceptable precision calculating the approximate solution. As it can be seen from \eqref{ExCEps} the solution $u_0$ is symmetric and satisfies $u_0(\frac\pi 2)=1$, meanwhile the best result we were able to obtain using \texttt{NDSolve} function was $u_0(\frac\pi 2)\approx 62$.

To apply the described algorithm we transformed the problem \eqref{CEeqn} to the interval $[0,1]$. We found that the Clenshaw--Curtis integration requires a lot of extra precision to evaluate the iterative integrals. The following simple test was used. It follows from the definition of the transmutation operator \eqref{Tmain} and Theorem \ref{Th Transmute} that $\frac{\varphi_k(x)}{x^k}\to 1$, $k\to\infty$. Hence, any weird behavior of the quantity $\frac{\varphi_k(x)}{x^k}$ when $k$ increases indicates serious errors in the calculated formal powers. On Figure \ref{ExCEphiError} we present the graphs of $\varphi_k(1)$ evaluated with different values of $M$ and different precision. We tried formulas \eqref{Xn}, \eqref{Xtiln} as well as recently discovered formulas from \cite{KrTNewSPPS} for the calculation of the formal powers. As one can see from the presented graphs, for the same value of the parameter $M$ and the same arithmetic precision the formulas from \cite{KrTNewSPPS} allow one to roughly double the number of calculated formal powers. Moreover, a further increase of the number of calculated formal powers can be achieved by increasing the arithmetic precision.

\begin{figure}[htb]
\centering
\includegraphics[
height=2.17in,
width=5.0in
]
{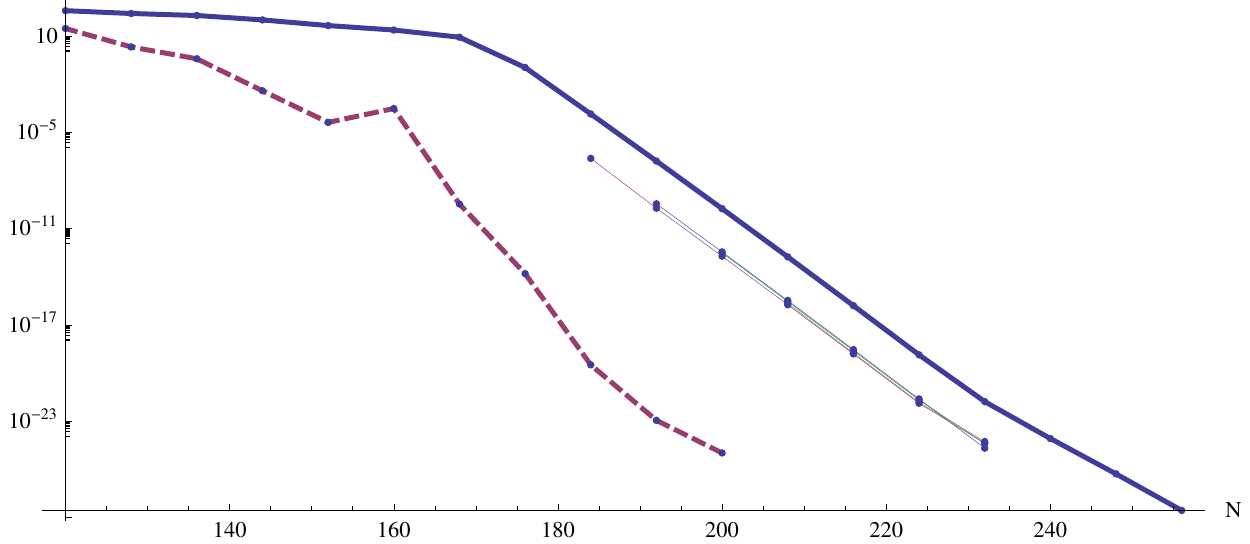}
\caption{The graphs of the maximal approximation error in \eqref{KxxErr}
and \eqref{KxmxErr} (solid thick line), the maximal absolute error of the isolated eigenvalues (dashed line) and absolute errors of the eigenvalues of the first 4 clusters (4 almost coinciding thin lines) of the Coffey-Evans problem (Example \ref{ExCE}) as functions of $N$.}
\label{ExCEEigError}
\end{figure}

For the first experiment we calculated the functions $\mathbf{c}_n$ and $\mathbf{s}_n$ for $n\le 256$ using $M=8192$ and 2500-digit arithmetic for the intermediate calculations, final values of the functions were stored at 1025 points with 400 digit precision. Our results were compared with those from \cite{AliciTaceli}. We calculated the error of the eigenvalues separately for the isolated eigenvalues and for each of the first four clusters. The obtained errors are presented on Figure \ref{ExCEEigError}. Note that the errors of the isolated eigenvalues decrease much faster than the errors of approximation in \eqref{KxxErr} and \eqref{KxmxErr}, meanwhile the slope of graph of the error for the clustered eigenvalues is an agreement with the slope of the error of approximation. Note that the lower points of the error graphs correspond to the limit of precision of the values presented in \cite{AliciTaceli}.

\begin{figure}[htb]
\centering
\includegraphics[
height=2.0in,
width=5.0in
]
{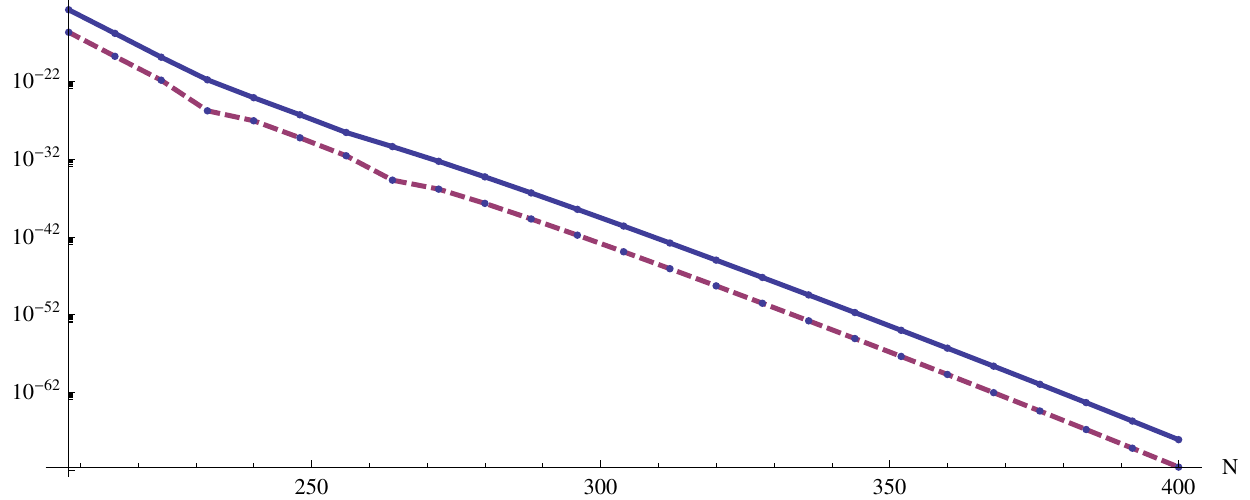}
\caption{The graphs of the maximal approximation error in \eqref{KxxErr}
and \eqref{KxmxErr} (solid line) and the absolute error of the eigenvalues of the first cluster (dashed line) of the Coffey-Evans problem (Example \ref{ExCE}) as functions of $N$.}
\label{ExCEEigError2}
\end{figure}

For the second experiment we calculated the functions $\varphi_n$, $\mathbf{c}_n$ and $\mathbf{s}_n$ for $n\le 400$ using $M=16384$ and 4000-digit arithmetic for the intermediate calculations. It is known \cite{KrTNewSPPS} that for the first eigenvalues the SPPS method achieves an especially remarkable accuracy. Hence we computed the eigenvalues of the first cluster using the SPPS representation with all functions $\varphi_n$, $n\le 400$ and used the obtained values to verify the accuracy of the proposed algorithm.
We observed an exponential decay of approximation errors in \eqref{KxxErr} and \eqref{KxmxErr} as well as of the error of the computed eigenvalues with respect to $N$. The slopes of both lines are nearly equal, see Figure \ref{ExCEEigError2}. For $N=380$ the errors of approximation were $2.12\cdot 10^{-63}$ and $2.16\cdot 10^{-63}$. In Table \ref{ExCETable1} we present the obtained eigenvalues. We are sure that they are exact for all 65 decimal places, evaluation with higher $N$ confirmed the presented digits. The eigenvalue with the index 0 is taken from the roots of the SPPS polynomial. On Figure \ref{ExCEEigenfunctions} we illustrate that our method allows one to obtain eigenfunctions as well.
\begin{table}[htb]
\centering
\small
\begin{tabular}
{cr@{}l}\hline
$n$ & \multicolumn{2}{c}{$\lambda_{n}$ (approximated)}\\\hline
0 &   4.&712 683 501 976 174 806 164 70 $\cdot 10^{-42}$\\
1 & 197.&968 726 516 507 291 450 189 104 613 631 137 680 282 238 501 965 516 499 678 457 445 33\\
2 & 391.&808 191 489 053 841 050 234 434 838 960 967 152 793 679 673 029 474 438 776 558 053 94\\
3 & 391.&808 191 489 053 841 832 241 250 450 567 879 442 934 703 990 750 980 000 552 054 926 43\\
4 & 391.&808 191 489 053 842 614 248 066 062 174 820 764 145 024 196 402 821 083 449 116 736 40\\
5 & 581.&377 109 231 579 654 864 715 898 934 768 731 234 409 061 366 366 202 824 138 138 986 73\\
6 & 766.&516 827 285 532 616 579 817 794 300 693 795 455 315 745 010 536 896 934 624 541 778 33\\
7 & 766.&516 827 285 535 505 431 430 237 728 556 528 324 964 223 414 154 644 684 143 065 548 50\\
8 & 766.&516 827 285 538 394 283 042 681 500 534 232 365 256 146 570 623 086 585 717 773 157 61\\
9 & 947.&047 491 585 860 179 592 142 658 200 615 670 883 560 237 084 089 403 375 253 394 471 82\\
10 &1122.&762 920 067 901 205 616 045 550 505 249 660 804 795 577 778 366 617 496 311 372 260 47\\
11 &1122.&762 920 071 056 526 891 891 942 465 507 589 782 179 421 839 584 709 544 874 366 683 38\\
12 &1122.&762 920 074 211 848 168 115 209 545 412 485 203 977 238 174 036 843 957 138 054 891 76\\
13 &1293.&423 567 331 707 081 413 958 872 197 134 275 865 126 916 380 700 329 000 293 361 954 74\\
14 &1458.&746 557 025 357 659 317 371 063 260 166 052 216 899 792 667 117 964 891 413 575 933 50\\
15 &1458.&746 558 472 128 708 810 534 887 553 090 428 313 351 825 225 479 463 874 997 007 515 49\\
16 &1458.&746 559 918 899 832 786 248 167 778 046 441 588 242 318 698 298 075 300 043 185 188 61\\
17 &1618.&391 008 042 643 345 932 885 816 053 496 039 220 613 799 984 685 321 316 858 406 202 35\\
18 &1771.&934 971 252 995 278 016 339 167 903 087 106 095 945 385 769 959 120 071 823 242 633 83\\
19 &1771.&935 290 604 372 265 020 106 948 865 313 215 142 762 994 141 327 857 394 637 946 848 95\\
20 &1771.&935 609 959 205 928 887 875 652 226 707 319 017 267 239 862 605 818 421 687 555 219 86\\
25 &2189.&490 124 838 400 777 638 432 025 527 998 260 500 900 985 813 188 172 572 445 106 470 28\\
50 &3928.&016 942 351 712 838 529 915 885 833 769 553 216 441 127 779 708 418 407 221 179 291 87\\
100&11470.&288 862 210 604 335 996 473 673 114 332 968 420 756 176 847 415 047 788 545 642 847 41\\
\hline
\end{tabular}
\caption{The eigenvalues of the Coffey-Evans problem (Example \ref{ExCE}) for $\beta=50$.}
\label{ExCETable1}%
\end{table}

\begin{figure}[htb]
\centering
\includegraphics[
height=1.242in,
width=5.0in
]
{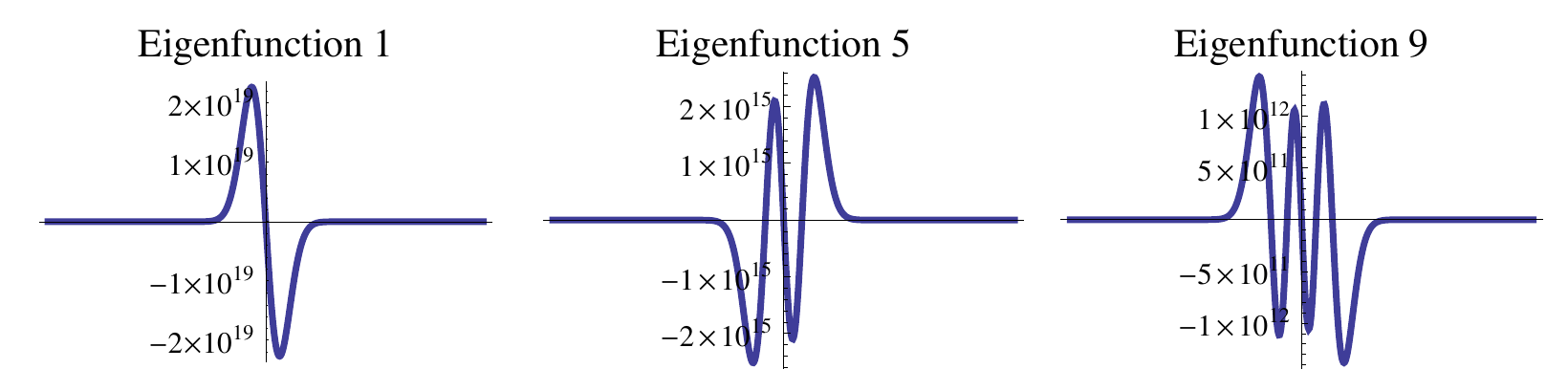}\\
\includegraphics[
height=1.242in,
width=5.0in
]
{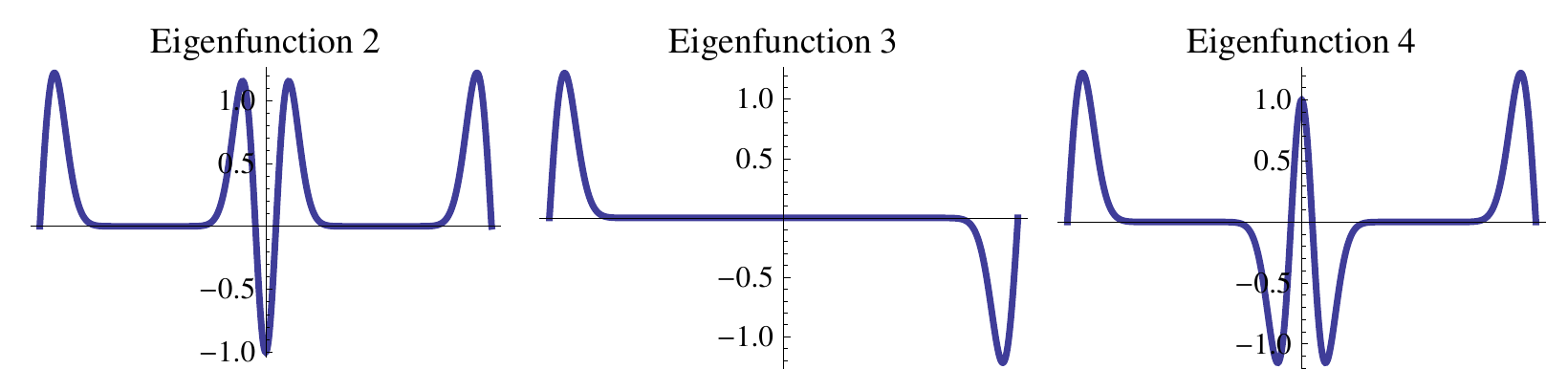}\\
\includegraphics[
height=1.242in,
width=5.0in
]
{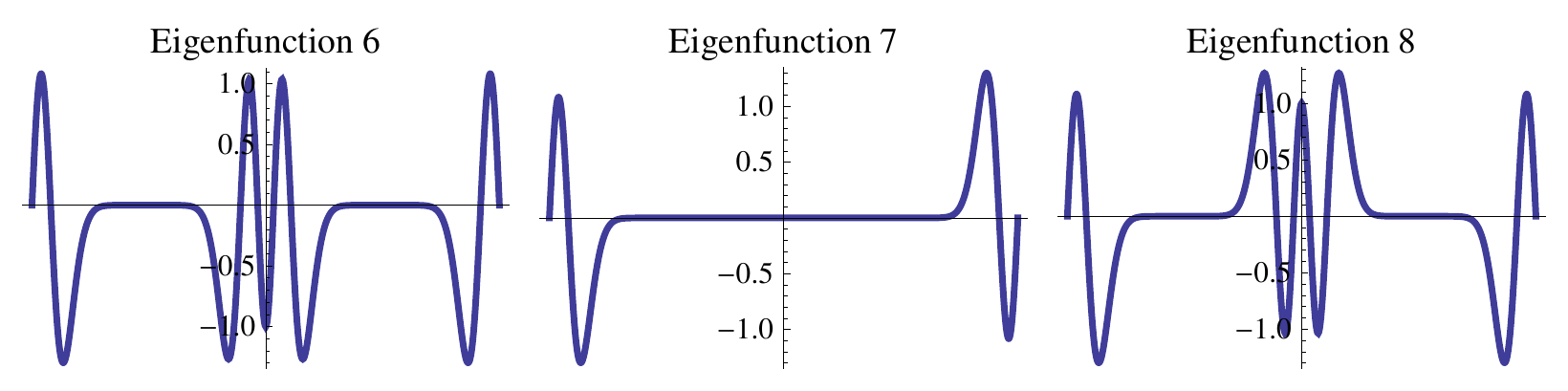}
\caption{The graphs of the eigenfunctions of the Coffey-Evans problem (Example \ref{ExCE}), on the top row corresponding to first 3 isolated eigenvalues, on the middle row corresponding to the eigenvalues from the first cluster and on the bottom row corresponding to the second cluster.}
\label{ExCEEigenfunctions}
\end{figure}
\end{example}

\subsection{Complex potential and spectral parameter dependent boundary conditions}
\begin{example}\label{ExBoumenir}
First we consider a problem with a complex potential for which the eigenvalues are known explicitly:
\begin{equation}
\label{BoumenirEqn}
\begin{cases}
-u''+(3+4i) u=\lambda u, & 0\le x\le \pi,\\
u'(0)=u'(\pi)=0.
\end{cases}
\end{equation}
A similar problem was treated in \cite{Boumenir2001}, \cite{Chanane2007}. We pose the Neumann boundary conditions to illustrate the performance of the proposed method in the case when one has to use approximations of both transmutation operators $T_f$ and $T_{1/f}$  and there is no zero coefficient in the expression \eqref{uprime}.

\begin{figure}[htb]
\centering
\includegraphics[
height=2.78in,
width=5.0in
]
{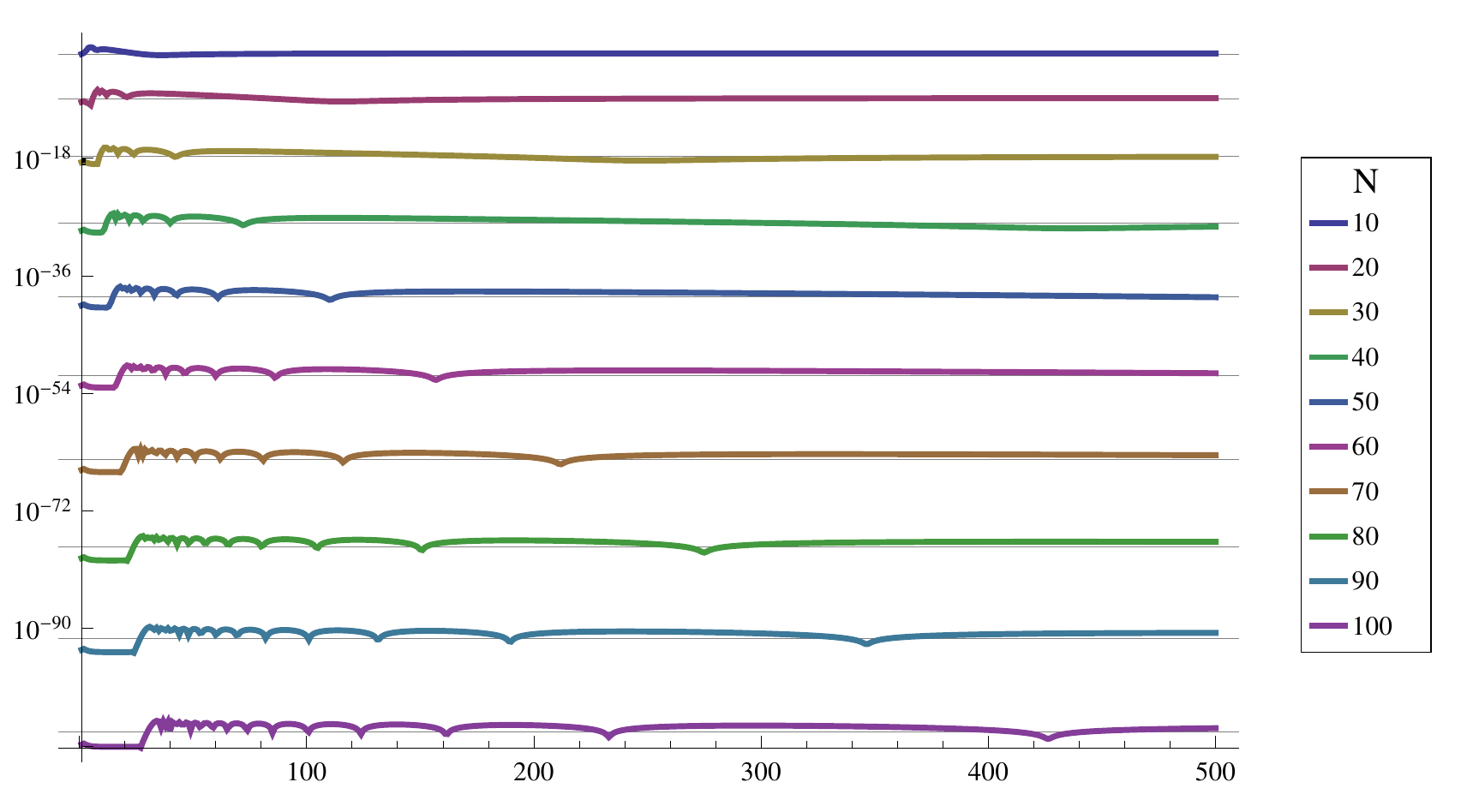}
\caption{The graphs of the absolute error of the first 500 eigenvalues of the problem \eqref{BoumenirEqn} (Example \ref{ExBoumenir}) obtained by the proposed algorithm with $N=10,20,\ldots,100$ (thick lines). Each thin line shows corresponding maximal approximation error in \eqref{KxxErr} and \eqref{KxmxErr}.}
\label{ExBoumenirError}
\end{figure}

The eigenvalues of the problem \eqref{BoumenirEqn} are given by $\lambda_n=n^2+3+4i$, $n=0,1,2\ldots$. As a particular solution we took $f(x)=e^{(2+i)x}$, then $h=f'(0)=2+i$, and the solution of the equation in \eqref{BoumenirEqn} satisfying the left boundary condition can be taken in the form $u(x,\omega):=c(\omega,x;h)-h s(\omega,x;\infty)$, see \eqref{ICcos} and \eqref{ICsin}. Equation $u'(\pi,\omega)=0$ gives us the characteristic equation of the problem \eqref{BoumenirEqn}.

For the numerical experiment we computed the functions $\mathbf{c}_n$ and $\mathbf{s}_n$ for $n\le 100$ using $M=256$ and 200-digit arithmetic for the calculation of the iterative integrals. After that we found the coefficients $a_0,a_1,\ldots,a_N$ and $b_1,\ldots,b_N$  for \eqref{KxxErr} and \eqref{KxmxErr} and approximated the integral kernel $K_{1/f}$ by \eqref{K1/f}. On Figure \ref{ExBoumenirError} we present the absolute error of the first 500 eigenvalues obtained by the proposed algorithm with different values of $N$. Note that the absolute errors remain at the same level and are in an excellent agreement with the approximation errors.
\end{example}

\begin{example}\label{ExChanane}
Consider the following problem with a complex potential and a spectral parameter dependent boundary condition \cite{Chanane2007}
\begin{equation}
\label{ChananeEqn}
\begin{cases}
-u''+e^{2ix} u=\mu^2 u, & 0\le x\le 1,\\
u(0)+\mu u(1)=0, & u'(0)=0.
\end{cases}
\end{equation}
For this example the characteristic equation is given by the equality
\[
\pi  \mu  \left(J_{-\mu }\left(e^i\right) (\mu  J_{\mu
   }(1)-J_{\mu -1}(1))+(J_{-\mu -1}(1)+\mu  J_{-\mu
   }(1)) J_{\mu }\left(e^i\right)\right)=2 \sin (\pi
   \mu ),
\]
where $J$ is the Bessel function of the first kind. As a particular solution we took  $f(x)=\frac{Y_0(e^{ix})}{Y_0(1)}$ with $f'(x)=-\frac{ie^{ix}Y_1(e^{ix})}{Y_0(1)}$ and hence $h:=f'(0)=-\frac{iY_1(1)}{Y_0(1)}$, where $Y$ is the Bessel function of the second kind.

We checked that the proposed algorithm was able to produce accurate results even using small values of the parameters $N$ and $M$. We used $N=20$, $M=96$ and 24-digit arithmetic. The approximation errors in \eqref{KxxErr} and \eqref{KxmxErr} were less than $2\cdot 10^{-18}$. The computation time was 1.6 seconds for the approximation part (Steps 1--5 of the algorithm from Subsection \ref{SubsectAlgorithm}) and 15 seconds was required to finish Step 6.
In Table \ref{ExChananeTable} we present the exact eigenvalues of the problem \eqref{ChananeEqn} together with the absolute errors obtained by our method and those reported in \cite{Chanane2007}.

\begin{table}[htb]
\centering
\small
\begin{tabular}
{cr@{+}lcc}\hline
$n$ & \multicolumn{2}{c}{$\mu^2_{n}$ (exact)} & \begin{tabular}{c}
Abs. error,\\
our method
\end{tabular} &
\begin{tabular}{c}
Abs. error,\\
method from \cite{Chanane2007}
\end{tabular}
\\\hline
1 & 4.9685430929323576232 & 0.3906545895360696300$i$ & $7.45\cdot 10^{-19}$ & $5.549\cdot 10^{-15}$ \\
2 & 20.602710348893372907 & 0.750232523531540313$i$ & $4.15\cdot 10^{-19}$ & $3.393\cdot10^{-14}$ \\
3 & 64.140382448045471607 & 0.684228375311332294$i$ & $3.84\cdot 10^{-19}$ & $3.977\cdot10^{-13}$\\
4 & 119.34792168887388950 & 0.714972404794013828$i$ & $3.88\cdot 10^{-19}$ & $8.004\cdot10^{-13}$\\
5 & 202.31443747778733950 & 0.70057212586524954$i$ & $2.43\cdot 10^{-19}$ & $2.064\cdot10^{-13}$\\
7 & 419.44558800598640866 & 0.70446189520144488$i$ & $2.39\cdot 10^{-17}$ & $4.582\cdot10^{-12}$\\
10 & 889.18520034251622114 & 0.70898948206981412$i$ & $1.05\cdot 10^{-17}$ & $2.734\cdot10^{-11}$\\
15 & 2077.5390063282081426 & 0.7073452595732362$i$ & $9.8\cdot 10^{-19}$ & $8.757\cdot10^{-10}$\\
20 & 3751.3714273572505215 & 0.7082902747558125$i$ & $6.69\cdot 10^{-18}$ & $2.410\cdot10^{-9}$\\
25 & 5926.6847018611521726 & 0.7078182266190013$i$ & $5.9\cdot 10^{-18}$ & $0.0003165$\\
50 & 24181.452786752952659 & 0.708107048845456$i$ & $3.6\cdot 10^{-18}$ &\\
75 & 54781.226477096285949 & 0.708045812501978$i$ & $4.5\cdot10^{-18}$ & \\
100 & 97710.005609281581148 & 0.708081740843471$i$ & $5.0\cdot 10^{-18}$ &\\
\hline
\end{tabular}
\caption{The eigenvalues of the problem \eqref{ChananeEqn} (Example \ref{ExChanane}) and the absolute eigenvalue errors obtained by our method and by method from \cite{Chanane2007}.}
\label{ExChananeTable}
\end{table}
\end{example}

\subsection{Quantum wells}
In this subsection we use notations and recall some results from \cite{CKOR}.

Consider the eigenvalue problem for the one dimensional Schr\"{o}dinger operator
\begin{equation}\label{eqQuantWell}
    Hu:=-u''+Q(x)u=\lambda u,\qquad x\in\mathbb{R},
\end{equation}
where
\begin{equation*}
    Q(x)=\begin{cases}
    \alpha_1, & x<0\\
    q(x), & 0\le x\le\ell,\\
    \alpha_2,& x>\ell
    \end{cases}
\end{equation*}
$\alpha_1$ and $\alpha_2$ are complex constants and $q$ is a continuous function on the segment $[0,\ell]$. The spectral problem consists in finding values of the spectral parameter $\lambda\in\mathbb{C}$ for which equation \eqref{eqQuantWell} possesses a nontrivial solution $u$ belonging to the Sobolev space $H^2(\mathbb{R})$.

In the selfadjoint case, i.e., when $Q$ is a real-valued function, the operator $H$ has a continuous spectrum $\left[\min\{\alpha_1,\alpha_2\},+\infty\right)$ and a discrete spectrum located on the set
\begin{equation}\label{DiscrSpectrQW}
    \Bigl[\min_{x\in[0,\ell]} q(x),\min\{\alpha_1,\alpha_2\}\Bigr).
\end{equation}

It was shown in \cite{CKOR} that finding the eigenvalues of the operator $H$ is equivalent to the Sturm-Liouville spectral problem for equation \eqref{eqQuantWell} on the segment $[0,\ell]$ with the boundary conditions
\begin{align}
u'(0)-\mu u(0)&=0,   \label{bcQWleft}\\
u'(\ell)+\nu u(\ell) &=0, \label{bcQWright}
\end{align}
where $\mu=+\sqrt{\alpha_1-\lambda}$, $\nu=\sqrt{\alpha_2-\lambda}$ and the eigenvalues are sought on the interval \eqref{DiscrSpectrQW}.

We consider a particular case $\alpha_1=\alpha_2=0$. Let $f$ be a particular solution of \eqref{eqQuantWell} for $\lambda=0$, non-vanishing on $[0,\ell]$ and satisfying $f(0)=1$. Define $h:=f'(0)$. Introducing a new parameter $\lambda=-\beta^2$ we obtain that $\mu=\nu=\beta$ in \eqref{bcQWleft} and \eqref{bcQWright},
\begin{equation*}
    \beta\in    \Bigl[0,\sqrt{\max_{x\in[0,\ell]} (-q(x))}\Bigr),
\end{equation*}
and the spectral parameter $\omega=i\beta$. It follows from \eqref{ICcos}, \eqref{ICsin} that the function $u(x)=c(\omega, x;h)-(i\omega+h)s(\omega, x;\infty)$ satisfies \eqref{bcQWleft}, and from \eqref{bcQWright} we obtain the characteristic equation
\begin{equation*}
    c'(\omega,\ell;h)-(i\omega+h)s'(\omega,\ell;\infty)-i\omega c(\omega,\ell;h)+i\omega(i\omega+h)s(\omega,\ell;\infty)=0.
\end{equation*}

\begin{example}\label{ExSquareWell}
Consider the square-well potential
\begin{equation*}
    Q(x)=\begin{cases}
    -U, & |x|\le a\\
    0, & \text{elsewhere}.
    \end{cases}
\end{equation*}
The eigenvalue $\lambda_n=-\beta_n^2$ is a solution of equation
\[
\arctan \frac{\sqrt{U-\beta_n^2}}{\beta_n}+a\sqrt{U-\beta_n^2}=\frac{n\pi}2,
\]
and the number of eigenvalues for each value of $U$ is equal to the smallest integer greater or equal to $2a\sqrt{U}/\pi$.

We have chosen the numerical values $U=15$, $a=1$, shifted the problem to the segment $[0,2]$ and taken $f(x)=e^{i\sqrt U x}$ as a particular solution. We used $N=32$, $M=96$ and 32-digits arithmetic. The errors in \eqref{KxxErr} and \eqref{KxmxErr} were of the magnitude $5.5\cdot 10^{-19}$. The exact eigenvalues of the problem and the absolute errors of the approximated eigenvalues are listed in Table \ref{ExSquareWellTable}.
\begin{table}[htb]
\centering
\small
\begin{tabular}
{cccccc}\hline
$n$ & $\beta_{n}$ (our method) & Absolute error & $\beta_{n}$ (\cite{CKOR}) & Absolute error
\\\hline
1 & 1.54436716376282718435 & $2\cdot 10^{-20}$ & 1.544367170 & $6\cdot 10^{-9}$\\
2 & 2.99547074607315853471 & $1.2\cdot 10^{-19}$ & 2.995470748 & $2\cdot 10^{-9}$\\
3 & 3.66781322275488144840 & $9\cdot 10^{-20}$ & 3.667813223 & $<3\cdot 10^{-10}$\\
\hline
\end{tabular}
\caption{Approximations of $\beta_n=\sqrt{-\lambda_n}$ of the square-well potential (Example \ref{ExSquareWell}).}
\label{ExSquareWellTable}
\end{table}
\end{example}

\begin{example}\label{ExSechSquare}
Consider the sech-squared potential defined by the expression $Q(x)=-m(m+1)\operatorname{sech}^2x$, $x\in(-\infty,\infty)$, $m\in\mathbb{N}$. An attractive feature of the potential $Q$ is that its eigenvalues can be calculated explicitly. The eigenvalue $\lambda_n$ is given by the formula $\lambda_n=-(m-n)^2$, $n=0,1,\ldots,m-1$.

The potential $Q$ is not of a finite support, nevertheless its absolute value decreases rapidly when $x\to\pm\infty$. We approximate the original problem by a problem with a finitely supported potential $\widehat Q$ defined by the equality
\begin{equation*}
    \widehat Q(x)=\begin{cases}
    -m(m+1)\operatorname{sech}^2x, & |x|\le a\\
    0, & \text{elsewhere},
    \end{cases}
\end{equation*}
where $a$ is chosen in such way that $Q(a)$ is sufficiently small.

For the numerical experiment we took $a=10$. Again in this example the recently discovered formulas from \cite{KrTNewSPPS} produced much more accurate results than formulas \eqref{Xn} and \eqref{Xtiln}. Using $N=70$, $M=2096$ and 128-digit arithmetic for the case $m=3$ we obtained the results presented in Table \ref{ExSechSqWellTable}.
\begin{table}[htb]
\centering
\small
\begin{tabular}
{cccc}\hline
$n$ & Exact values & $\lambda_{n}$ (our method) & $\lambda_{n}$ (\cite{CKOR})
\\\hline
0 & $-9$ & $-8.99999999999999999980$ & $-8.999628656$\\
1 & $-4$ & $-4.00000000000000000020$ & $-3.999998053$\\
2 & $-1$ & $-0.99999999999999877643$ & $-0.999927816$\\
\hline
\end{tabular}
\caption{Approximations of $\lambda_n$ of the potential $-12\operatorname{sech^2}x$ (Example \ref{ExSechSquare}).}
\label{ExSechSqWellTable}
\end{table}
\end{example}

\appendix
\section{Proof of Theorem \ref{Th Estimate for K1/f}}

The proof is based on several results from \cite{KT Transmut}, so we
preserve notations from \cite{KT Transmut}. Consider the bicomplex function $%
W:=\mathbf{K}_{f}-\mathbf{j}\mathbf{K}_{1/f}$. Here $\mathbf{j}$ is the hyperbolic
imaginary unit: $\mathbf{j}^{2}=1$. Using (\ref{K(x,t)}) and (\ref{K1/f}) we
observe that $W_{N}:=K_{f,N}-\mathbf{j}K_{1/f,N}=\sum_{n=0}^{N}Z^{(n)}(%
\alpha _{n},0;z)$ where $\alpha _{n}:=a_{n}+\mathbf{j}b_{n}$ and $Z^{(n)}$
are hyperbolic pseudoanalytic formal powers admitting the representation
\cite{KT Transmut}%
\begin{align*}
Z^{(0)}(\alpha _{0},0;z)& =a_{0}u_{0}(x,t)+\mathbf{j}b_{0}v_{0}(x,t)=\frac{h%
}{2}\left( f(x)+\frac{\mathbf{j}}{f(x)}\right) , & &  \\ 
Z^{(n)}(\alpha _{n},0;z)& =a_{n}u_{2n-1}(x,t)+b_{n}u_{2n}(x,t)+\mathbf{j}%
\bigl(a_{n}v_{2n}(x,t)+b_{n}v_{2n-1}(x,t)\bigr), & & n\geq 1.
\end{align*}
Here \ the hyperbolic variable $z$ has the form $z=x+\mathbf{j}t$ with the
corresponding conjugate $\overline{z}=Cz:=x-\mathbf{j}t$.

Denote $\widetilde{\mathbf{K}}_{f}=T_{f}^{-1}\left[ \mathbf{K}_{f}\right] $
and $\widetilde{K}_{f,N}=T_{f}^{-1}\left[ K_{f,N}\right] $. For the
corresponding Goursat data we introduce the notations%
\begin{equation*}
\begin{pmatrix}
\varphi (x) \\
\psi (x)%
\end{pmatrix}%
:=%
\begin{pmatrix}
\widetilde{\mathbf{K}}_{f}(x,x) \\
\widetilde{\mathbf{K}}_{f}(x,-x)%
\end{pmatrix}%
\quad \text{and}\quad
\begin{pmatrix}
\varphi _{N}(x) \\
\psi _{N}(x)%
\end{pmatrix}%
:=%
\begin{pmatrix}
\widetilde{\mathbf{K}}_{f,N}(x,x) \\
\widetilde{\mathbf{K}}_{f,N}(x,-x)%
\end{pmatrix}%
.
\end{equation*}%
Then $\left\vert \varphi -\varphi _{N}\right\vert \leq \varepsilon
\Vert T_{f}^{-1}\Vert $ and $\left\vert \psi -\psi
_{N}\right\vert \leq \varepsilon \Vert T_{f}^{-1}\Vert $. For an
estimate of the uniform norm $\Vert T_{f}^{-1}\Vert $ we refer to
\cite{KKTT}. It depends on $f$ and $b$ only.

Let us consider the functions $\Phi (x):=2\varphi (x)$, $\Psi (x)=2\psi
(x)-h $, $\Phi _{N}(x):=2\varphi _{N}(x)$, $\Psi _{N}(x)=2\psi _{N}(x)-h$.
We have that $\widetilde{W}:=\widetilde{\mathbf{K}}_{f}-\mathbf{j}\widetilde{%
\mathbf{K}}_{1/f}$ and $\widetilde{W}_{N}:=\widetilde{K}_{f,N}-\mathbf{j}%
\widetilde{K}_{1/f,N}$ are the unique solutions of the following
corresponding Goursat problem \cite{KT Transmut}%
\begin{equation*}
\partial _{\bar{z}}\widetilde{W}=0,
\end{equation*}%
\begin{equation*}
\widetilde{W}(x,x)=P^{+}\Phi \left( x\right) +P^{-}\Psi \left( 0\right)
\quad \text{and\quad }\widetilde{W}(x,-x)=P^{+}\Phi \left( 0\right)
+P^{-}\Psi \left( x\right)
\end{equation*}%
and
\begin{equation*}
\partial _{\bar{z}}\widetilde{W}_{N}=0,
\end{equation*}%
\begin{equation*}
\widetilde{W}_{N}(x,x)=P^{+}\Phi _{N}\left( x\right) +P^{-}\Psi _{N}\left(
0\right) \quad \text{and\quad }\widetilde{W}_{N}(x,-x)=P^{+}\Phi _{N}\left(
0\right) +P^{-}\Psi _{N}\left( x\right)
\end{equation*}%
where the idempotents $P^{+}$ and $P^{-}$ are defined by $P^{\pm }=\frac{1}{2%
}(1\pm \mathbf{j})$. Moreover, $\widetilde{W}=\mathbf{V}_{1}^{-1}[
\mathbf{K}_{f}-\mathbf{jK}_{1/f}] $ where the operators $\mathbf{V}%
_{1} $ and $\mathbf{V}_{1}^{-1}$ were introduced in \cite{KT Transmut} as $%
\mathbf{V}_{1}^{\pm 1}=T_{f}^{\pm 1}\mathcal{R}+\mathbf{j}T_{1/f}^{\pm 1}%
\mathcal{I}$ with $\mathcal{R}$ and $\mathcal{I}$ being projection operators
projecting a bicomplex valued function onto the respective scalar
components, $\mathcal{R=}\frac{1}{2}(I+C)$ and $\mathcal{I}=\frac{1}{2%
\mathbf{j}}(I-C)$ where $I$ is the identity operator. Thus, $\widetilde{W}%
=T_{f}^{-1}\left[ \mathbf{K}_{f}\right] -\mathbf{jT}_{1/f}^{-1}\left[
\mathbf{K}_{1/f}\right] $ meanwhile for $\widetilde{W}_{N}$ we have $%
\widetilde{W}_{N}=\mathbf{V}_{1}^{-1}\left[ \sum_{n=0}^{N}Z^{(n)}(\alpha
_{n},0;z)\right] =\sum_{n=0}^{N}\alpha _{n}z^{n}$ (see \cite[Sect. 5]{KT
Transmut}). Now we consider
\begin{equation}
\Vert W-W_{N}\Vert =\Vert \mathbf{V}_{1}\widetilde{W}-%
\mathbf{V}_{1}\widetilde{W}_{N}\Vert \leq \Vert \mathbf{V}%
_{1}\Vert \Vert \widetilde{W}-\widetilde{W}_{N}\Vert
\label{estimate1}
\end{equation}
where $\Vert \widetilde{W}-\widetilde{W}_{N}\Vert =\max_{%
\overline{S}}\vert \widetilde{W}-\widetilde{W}_{N}\vert _{\mathbb{%
B}}$ and the norm $\left\vert w\right\vert _{\mathbb{B}}$ of a bicomplex
number $w$ is defined as in \cite[Sect. 5]{KT Transmut}: $\left\vert
w\right\vert _{\mathbb{B}}=\frac{1}{2}\left( \left\vert w^{+}\right\vert
+\left\vert w^{-}\right\vert \right) $ and $w^{+}$, $w^{-}\in \mathbb{C}$
are such that $w=P^{+}w^{+}+P^{-}w^{-}$. In \cite[Sect. 5]{KT Transmut} it
was shown that $\mathbf{V}_{1}$ and $\mathbf{V}_{1}^{-1}$ are bounded in the
space of continuous bicomplex valued functions on $\overline{S}$ with the
norm $\left\Vert \cdot \right\Vert $ and hence $\left\Vert \mathbf{V}%
_{1}\right\Vert $ is a finite number depending on $f$ and $b$ only.

In a full analogy with the proof of Proposition 2.10 from \cite[Sect. 5]{KT
Transmut} we obtain $\Vert \widetilde{W}-\widetilde{W}_{N}\Vert
<\varepsilon \Vert T_{f}^{-1}\Vert $. Then from (\ref{estimate1})
we have
\begin{equation*}
\Vert \mathbf{K}_{f}-\mathbf{jK}_{1/f}-K_{f,N}+\mathbf{j}%
K_{1/f,N}\Vert <\varepsilon \Vert T_{f}^{-1}\Vert
\Vert \mathbf{V}_{1}\Vert .
\end{equation*}%
Finally, since $\left\vert \mathcal{I}w\right\vert \leq \left\vert
w\right\vert _{\mathbb{B}}$ (see \cite[Proposition 2]{CKr2012}), we obtain
\[
\vert \mathbf{K}_{1/f}-K_{1/f,N}\vert \leq \Vert \mathbf{K}%
_{f}-\mathbf{jK}_{1/f}-K_{f,N}+\mathbf{j}K_{1/f,N}\Vert <\varepsilon
\Vert T_{f}^{-1}\Vert \Vert \mathbf{V}_{1}\Vert
\]
which finishes the proof.

\section*{Acknowledgements}
We thank our colleague R. Michael Porter for his valuable help with Wolfram Mathematica.

\end{document}